\documentclass[reqno]{amsart}
\usepackage{amsmath,amssymb,cite}
\usepackage[mathscr]{euscript}
\usepackage{graphics,epsfig,color}

\theoremstyle{plain}
\newtheorem{theorem}{Theorem}[section]
\newtheorem{proposition}[theorem]{Proposition}
\newtheorem{lemma}[theorem]{Lemma}
\newtheorem{corollary}[theorem]{Corollary}
\theoremstyle{definition}

\theoremstyle{remark}

\numberwithin{equation}{section}
\numberwithin{theorem}{section}

\newcommand{\mc}[1]{{\mathcal #1}}

\newcommand{\bs}[1]{{\boldsymbol #1}}
\newcommand{\bb}[1]{{\mathbb #1}}

\newcommand{\rmd}{\mathrm{d}}
\newcommand{\eps}{\varepsilon}

\newcommand{\id}{{1 \mskip -5mu {\rm I}}}
\newcommand{\supp}{\mathop{\rm supp}\nolimits}

\title[Vanishing viscosity limit for  concentrated vortex rings]{Vanishing viscosity limit for  concentrated vortex rings}

\author[P.\ Butt\`a]{Paolo Butt\`a}
\address{Dipartimento di Matematica\\
Sapienza Universit\`a di Roma\\
P.le Aldo Moro 5, 00185 Roma\\
Italy}

\email{butta@mat.uniroma1.it}

\author[G.\ Cavallaro]{Guido Cavallaro}
\address{Dipartimento di Matematica\\
Sapienza Universit\`a di Roma\\
P.le Aldo Moro 5, 00185 Roma\\
Italy}

\email{cavallar@mat.uniroma1.it}

\author[C.\ Marchioro]{Carlo Marchioro}
\address{Dipartimento di Matematica\\
Sapienza Universit\`a di Roma\\
P.le Aldo Moro 5, 00185 Roma\\
Italy\\
and\\
International Research Center M\&MOCS\\ 
Universit\`a di L'Aquila\\
Palazzo Caetani\\
04012 Cisterna di Latina (LT)\\
Italy}

\email{marchior@mat.uniroma1.it}

\keywords{Incompressible viscous flow, vortex rings.}

\subjclass{
76D17, 
37N10. 
}

\date{}

\begin{document}

\begin{abstract}
We study the time evolution of a viscous incompressible fluid with axial symmetry without swirl, when the initial vorticity is very concentrated in $N$ disjoint rings. We show that in a suitable joint limit, in which both the thickness of the rings and the viscosity tend to zero, the vorticity remains concentrated in $N$ disjointed rings, each one of them performing a simple translation along the symmetry axis with constant speed.
\end{abstract}

\maketitle

\section{Introduction}
\label{sec:1}

Axisymmetric flows without swirl are a special class of incompressible fluid flows in the whole space $\bb R^3$, having the advantage of being simple enough to be analyzed with mathematical rigor and, nevertheless, containing non-trivial and interesting examples. In particular, the existence of a global solution both for the Euler equation and the Navier–Stokes one has been established many years ago in Refs.~\cite{Lad,UY}, see also Refs.~\cite{Fe-Sv,Gal13,Gal12} for more recent results. 

The axisymmetric solutions without swirl are called sometimes \textit{vortex rings}, because in the ideal (i.e., inviscid) case there exist particular solutions whose shape remains constant in time (the so-called steady vortex rings) and translate along the symmetry axis with constant speed (the propagation velocity). The knowledge of such solutions is very old, but a first rigorous proof of their existence and properties goes back to Refs.~\cite{Fr70,FrB74, AmS89},  by means of variational methods. Other information and references on axially symmetric solution without swirl can be found in Ref.~\cite{ShL92}.

Besides these exceptional solutions, a more physically practical issue concerns the case of generic initial data such that the vorticity is supported in an annulus. This problem was first considered in Ref.~\cite{BCM00}, where in the case of Euler equation it is proved that if the initial vorticity is sharply concentrated then the motion is similar to that of a steady vortex ring. More precisely, if the vorticity is supported in an annulus of thickness $\eps$ and fixed radius, and the vorticity mass vanishes as $|\log\eps|^{-1}$, then in the limit $\eps\to 0$ the vorticity remains concentrated in a thin annulus which performs a rectilinear motion along the symmetry axis. We remark that the vorticity mass must vanish appropriately to have a nontrivial limit motion, as suggested by the translation velocity of the steady vortex ring with unitary vorticity mass, which is known to diverge in the limit of sharp concentration. This result has been extended to the case of Navier-Stokes equation in Ref.~\cite{BruMar}, in a simultaneous limit of vanishing viscosity and sharp concentration.

In this paper, we consider the more general case of an incompressible viscous fluid with axial symmetry without swirl in which the initial vorticity is concentrated in $N>1$ disjoint rings of thickness $\eps$, fixed radii and vorticity mass of the order of $|\log\eps|^{-1}$. As in the case of Ref.~\cite{BruMar}, we study the system in the joint limit in which the viscosity $\nu \to 0$ and simultaneously $\eps\to 0$, in such a way that $\nu \le \eps^2 |\log\eps|^\gamma$ (with $\gamma>0$). We show that the vorticity remains concentrated around $N$ rings, which perform uniform rectilinear motions parallel to the symmetry axis. 

This analysis constitutes a generalization of Refs.~\cite{BuM2,ButCavMar}, where the case without viscosity was studied. With respect to the case of a vortex alone considered in Ref.~\cite{BCM00}, the techniques of Refs.~\cite{BuM2,ButCavMar} allow not only to show that large part of the mass of each vortex ring remains concentrated in an annulus, but also that their supports do not overlap. This strong localization property plays a key role in showing that the interaction among the rings becomes negligible in the limit $\eps\to 0$ (for an application of these ideas to the lake equations see Ref.~\cite{HLM}). The extension to the present situation is therefore nontrivial, because the diffusive effect due to the viscosity produces an instantaneous spreading of the supports of the vorticities rings (initially with compact supports) in the whole space, thus overlapping with each other and making the mutual interaction more relevant. These technical issues will be carefully discussed later on.

There are other examples in the literature of special classes of concentrated initial data whose time evolution is closely related to the dynamics of particular systems of particles. We mention the case of an incompressible fluid with planar symmetry, in which the initial vorticity is concentrated in $N$ small disks, each one containing an amount of vorticity of order one. This system has been satisfactorily understood, obtaining the convergence to the celebrated \textit{point vortex model}, both in the Euler case, see  Refs.~\cite{BuM1,MaP93,MaP94,Mar98}, and in the vanishing viscosity limit, see Refs.~\cite{Mar90, Mar98,Gal11,CS}. 

Clearly, planar symmetry is a good approximation of axisymmetry without swirl when the fluid is observed far away from the symmetry axis. With this respect, we mention that the case of vortex rings with very large radii have been already considered in the literature. More precisely, in Ref.~\cite{Mar99} it is proved that in an inviscid system of $N$ vortex rings with thickness $\eps$, vorticity mass of order one, and radii of order $\eps^{-\alpha}$, $\alpha>0$, the centers of vorticity converge as $\eps\to 0$ to the solution of the point vortex system (the case of vanishing viscosity was treated later in Ref.~\cite{Mar07}). This result has been extended in Ref.~\cite{CavMar21} up to the case of radii of the order $|\log\eps|^\alpha$, with $\alpha>2$. The more interesting regime is for radii of order $ |\log\eps|$, where the limiting motion of the vorticity centers is conjectured to be governed by a dynamical system which is a combination of the point vortex dynamics plus simple translations with constant speed. For this problem, only the simple case of a vortex ring alone can be analyzed both for the Euler case (see Ref.~\cite{MarNeg}) and for Navier-Stokes (not present in the literature, but it can be proved along the same lines of Ref.~\cite{BruMar}). More challenging is the case of $N>1$ vortex rings, in which no results are available both for Euler and Navier-Stokes equations, and new ideas are needed since the basic tools used up to now fail in this critical regime.

\section{Notation and main result}
\label{sec:2}

The evolution of an incompressible viscous fluid of unitary density filling the whole space $\bb R^3$ is governed by the Navier-Stokes equations,
\begin{equation}
\label{eule}
\partial_t  \bs u + (\bs u\cdot \nabla) \bs u  =  -\nabla p +\nu\Delta \bs u\,, \qquad \nabla \cdot \bs u = 0 \,,
\end{equation}
where $\bs u = \bs u(\bs\xi,t)$ and $p=p(\bs\xi,t)$ are the velocity and pressure respectively, $\nu\ge 0$ is the viscosity, $\bs\xi = (\xi_1,\xi_2,\xi_3)$ denotes a point in $\bb R^3$, and $t\in \bb R_+$ is the time. More precisely, the evolution is determined by the Cauchy problem associated to \eqref{eule} with assigned boundary conditions. In what follows, we always assume that $\bs u$ decays at infinity (this happens under the initial conditions specified hereafter, see for instance Ref.~\cite{Fe-Sv}). Introducing the vorticity $\bs\omega$, defined by
\begin{equation}
\label{vort}
\bs\omega = \nabla \wedge \bs u \,,
\end{equation}
this implies that the velocity $\bs u$ can be reconstructed from $\bs\omega$ as
\begin{equation}
\label{u-vort}
\bs u(\bs\xi,t) = - \frac{1}{4\pi} \int\! \rmd \bs\eta \, \frac{(\bs\xi-\bs\eta) \wedge \bs\omega(\bs\eta,t)}{|\bs\xi-\bs\eta|^3} \,.
\end{equation}
On the other hand, by Eqs.~\eqref{eule} and \eqref{vort}, the vorticity evolves according to the equation
\begin{equation}
\label{vorteq}
\partial_t \bs\omega + (\bs u\cdot \nabla) \bs\omega  = (\bs \omega\cdot \nabla) \bs u +\nu\Delta \bs \omega \,,
\end{equation}
so that Eqs.~\eqref{u-vort} and \eqref{vorteq} give a formulation of the Navier-Stokes equations (with velocity decaying at infinity) in terms of the vorticity $\bs\omega$.

Denoting by $(z,r,\theta)$ the cylindrical coordinates, we recall that the vector field $\bs F$ of cylindrical components $(F_z, F_r, F_\theta)$ is called axisymmetric without swirl if $F_\theta=0$ and $F_z$ and $F_r$ are independent of $\theta$. 

The axisymmetry is preserved by the evolution equation Eq.~\eqref{eule}. Moreover, when restricted to axisymmetric velocity fields $\bs u(\bs\xi,t) = (u_z(z,r,t), u_r(z,r,t), 0)$, Eqs.~\eqref{vort} and \eqref{vorteq} expressed in cylindrical components reduce to
\begin{equation}
\label{omega}
\bs\omega = (0,0,\omega_\theta) = (0,0,\partial_z u_r - \partial_r u_z)
\end{equation}
and, denoting henceforth $\omega_\theta$ by $\omega$, 
\begin{equation}
\label{omeq}
\partial_t \omega + (u_z\partial_z + u_r\partial_r) \omega - \frac{u_r\omega}r = 
\nu \left[ \partial_z^2 \omega +\frac{1}{r}\partial_r\left( r\partial_r\omega\right) -\frac{\omega}{r^2}\right].
\end{equation}
We also notice that the solenoidal condition $\nabla \cdot \bs u = 0$ reads
\[
\partial_z(ru_z) + \partial_r(ru_r) = 0\,.
\]
Furthermore, in view of Eq.~\eqref{u-vort}, $u_z = u_z(z,r,t)$ and $u_r=u_r(z,r,t)$ are given by
\begin{align}
\label{uz}
u_z & = - \frac{1}{2\pi} \int\! \rmd z' \!\int_0^\infty\! r' \rmd r' \! \int_0^\pi\!\rmd \theta \, \frac{\omega(z',r',t) (r\cos\theta - r')}{[(z-z')^2 + (r-r')^2 + 2rr'(1-\cos\theta)]^{3/2}} \,,
\\ \label{ur}
u_r & = \frac{1}{2\pi} \int\! \rmd z' \!\int_0^\infty\! r' \rmd r' \! \int_0^\pi\!\rmd \theta \, \frac{\omega(z',r',t) (z - z')\,\cos \theta}{[(z-z')^2 + (r-r')^2 + 2rr'(1-\cos\theta)]^{3/2}} \,.
\end{align}
Hence, the axisymmetric solutions to the Euler equations are the solutions to Eqs.~\eqref{omeq}, \eqref{uz}, and \eqref{ur}.

We further observe that Eq.~\eqref{omeq} expressed in term of the quantity $\omega/r$ reads
\begin{equation}
\label{cons-omr}
\left[\partial_t + u_z\partial_z + \left(u_r - \frac{3\nu}{r} \right)\partial_r  \right] \left(\frac{\omega}{r}\right)  = \nu \left(\partial_z^2+\partial_r^2\right) \left(\frac{\omega}{r} \right),
\end{equation}
which means that $\omega/r$ evolves as a diffusion with drag. Finally, we notice that, by a formal integration by parts, the following weak formulation of Eq.~\eqref{omeq} holds true, 
\begin{equation}
\label{weq}
\frac{\rmd}{\rmd t} \omega_t[f] = \omega_t\left[u_z\partial_z f + u_r\partial_r f + \partial_t f\right]
+\nu \omega_t \left[\partial_z^2 f +\partial_r^2 f -\frac{1}{r}\partial_r f\right],
\end{equation}
where
\[
\omega_t[f] := \int\! \rmd z \!\int_0^\infty\! \rmd r \, \omega(z,r,t) f(z,r,t) \,,
\]
and $f = f(z,r,t)$ is any  smooth test function, such that the boundary terms in the integration by parts vanish (at $r=0$ and $r=+\infty$).

Denoting by $\Pi:=\left\{(z,r):r>0\right\}$ the open half-plane (note that a point in the half-plane $\Pi$ corresponds to a circumference in the three dimensional space $\bb R^3$), we take initial data for which the vorticity has compact support contained in $N$ disks in $\Pi$, i.e.,
\begin{equation}
\label{in}
\omega(z,r,0) = \sum_{i=1}^N  \omega_{i,\eps}^0(z,r) \,,
\end{equation}
being $\eps\in (0,1)$ a small parameter, where $\omega_{i,\eps}^0(z,r)$, $i=1,\ldots, N$, are functions with definite sign whose support is contained in $\Sigma(\zeta^i|\eps)$, which is the open disk of center $\zeta^i$ and radius $\eps$,
\begin{equation}
\label{initial}
\Lambda_{i,\eps}(0) := \supp\, \omega_{i,\eps}(\cdot,0) \subset \Sigma(\zeta^i|\eps)\,,
\end{equation}
with 
\[
\overline{\Sigma(\zeta^i|\eps)} \subset\Pi\quad \forall\, i\,, \qquad  \Sigma(\zeta^i|\eps)\cap \Sigma(\zeta^j|\eps)=\emptyset\quad \forall\, i \ne j\,,
\]
for fixed $\zeta^i = (z_i,r_i)\in \Pi$. We assume also that
\begin{equation}
\label{2D}
\min_i r_i > 2D\quad \forall\, i\,, \qquad |r_i-r_j| \ge 2D \quad \forall\, i \ne j\,,
\end{equation}
where $D$ is a positive fixed constant. This means that the annuli have different radii, which is an essential hypothesis for our analysis. Clearly, at positive time this separation is no more true, since by Eq.~\eqref{omeq} each initial vortex ring is spread in the whole space. Nevertheless, the decomposition Eq.~\eqref{in} extends also at positive times,
\begin{equation}
\label{in-t}
\omega_\eps(z,r,t) = \sum_{i=1}^N  \omega_{i,\eps}(z,r,t)\,,
\end{equation}
provided $\omega_{i,\eps}(z,r,t)$, $i=1,\ldots,N$, satisfy
\begin{equation}
\label{omeq-i}
\begin{split}
& \partial_t \omega_{i,\eps} + (u_z\partial_z + u_r\partial_r) \omega_{i,\eps} - \frac{u_r\omega_{i,\eps}}r = 
\nu \left[ \partial_z^2 \omega_{i,\eps} +\frac{1}{r}\partial_r\left( r\partial_r\omega_{i,\eps}\right) -\frac{\omega_{i,\eps}}{r^2}\right], \\ & \omega_\eps(z,r,0) = \omega_{i,\eps}^0(z,r)\,,
\end{split}
\end{equation}
where $(u_z,u_r)$ is the velocity field associated to the whole vorticity $\omega_\eps(z,r,t)$ by means of Eqs.~\eqref{uz} and \eqref{ur}.  

In order to have non trivial (neither vanishing nor diverging) limiting velocities of the vortex rings as $\eps\to 0$, the initial data have to be chosen appropriately.

As in the Euler case (absence of viscosity) analyzed in the quoted papers, the correct choice can be deduced by considering the simplest case of a vortex ring alone, of intensity $N_\eps = \int\! \rmd z \int _0^\infty\! \rmd r \, \omega_\eps(z,r,0)$ and supported in a small region of diameter $\eps$. It is well known that it moves along the $z$-direction with an approximately constant speed proportional to $N_\eps |\log \eps|$, see Ref.~\cite {Fr70}. With this in mind, we assume that there are $N$ real parameters $a_1,\ldots, a_N$, called \textit{vortex intensities}, such that
\begin{equation}
\label{ai}
|\log\eps| \int\!\rmd z \!\int_0^\infty\!\rmd r\, \omega_{i,\eps}(z,r,0) = a_i \quad \forall\,i=1,\ldots,N \,.
\end{equation}
Finally, to avoid too large vorticity concentrations, we further assume there is a constant $M>0$ such that 
\begin{equation}
\label{Mgamma}
|\omega_{i,\eps}(z,r,0)| \le \frac{M}{\eps^2|\log\eps|} \quad \forall\, (z,r)\in \Pi \quad \forall\, i=1,\ldots,N\,.
\end{equation}
Now, we can state the main result of the paper.

\begin{theorem}
\label{thm:1}
Assume the initial datum $\omega(z, r,0)$ verifies Eqs.~\eqref{in}, \eqref{initial}, \eqref{2D}, \eqref{ai}, and \eqref{Mgamma}, and define
\begin{equation}
\label{free_mi}
\zeta^i(t) :=  \zeta^i + \frac{a_i}{4\pi r_i} \begin{pmatrix} 1 \\ 0 \end{pmatrix} t \,, \quad i=1,\ldots,N\,.
\end{equation}
Then, for any $T>0$ the following holds true. For any $\eps$ small enough and $\nu\le \eps^2|\log\eps|^\gamma$, with $\gamma\in (0,1)$,  there are $\zeta^{i,\eps}(t)\in \Pi$, $t\in [0,T]$, $ i=1,\ldots,N$, and $R_\eps>0$ such that
\begin{equation}
\label{conc_thm1.1}
\lim_{\eps \to 0}|\log\eps| \int_{\Sigma(\zeta^{i,\eps}(t)|R_\eps)}\!\rmd z\,\rmd r\, \omega_{i,\eps}(z,r,t) = a_i\quad \forall\, i=1,\ldots,N, \quad \forall\, t\in [0,T]\,,
\end{equation}
with
\[
\lim_{\eps\to 0} R_\eps = 0, \qquad \lim_{\eps\to 0} \zeta^{i,\eps}(t) = \zeta^i(t) \quad \forall\, t\in [0,T]\,,
\]
where $\omega(z,r,t)$ is the time evolution of $\omega(z,r,0)$ via the Navier-Stokes equations.
\end{theorem}

We remark that the quite strong assumption $\nu\le \eps^2|\log\eps|^\gamma$ is needed to control the energy dissipated by viscosity, see Appendix \ref{sec:B}. This seems an  unavoidable condition to obtain the concentration result Eq.~\eqref{conc_thm1.1}, which requires a very small time variation of the energy. 

As for the Euler case studied in Ref.~\cite{ButCavMar}, because of Eq.~\eqref{omeq-i}, $\omega_{i,\eps}(z,r,t)$ can be viewed as the evolution of a single vortex ring driven by the sum of the velocity generated by $\omega_{i,\eps}$ itself plus an external time-depending field (the sum of the velocities generated by $\omega_{j,\eps}$ for $j\ne i$). Therefore, the proof of Theorem \ref{thm:1} will be obtained as a corollary (based on a bootstrap argument) of the analogous result for a modified system, which describes the motion of a single vortex in an external field. The analysis of this system is the content of the next two sections, where the parts which can be treated as in Ref.~\cite{ButCavMar} are only sketched and postponed to the Appendices.

\section{A single vortex in an external field}
\label{sec:3}

It is convenient to rename the variables by setting
\begin{equation}
\label{nv}
x = (x_1,x_2) := (z,r)\,
\end{equation}
and extend the vorticity to a function on the whole plane by taking $\omega_\eps(x,t)=0$ for $x_2\le 0$. Accordingly, the velocity field $(u_z,u_r)$ is denoted by $u = (u_1, u_2)$ and Eqs.~\eqref{uz} and \eqref{ur} take the form
\begin{equation}
\label{u=}
u(x,t) = \int\!\rmd y\, H(x,y)\, \omega_\eps(y,t)\,,
\end{equation}
with kernel $H(x,y) = (H_1(x,y),H_2(x,y))$ given by
\begin{align}
\label{H1}
H_1(x,y) & = \frac{1}{2\pi} \int_0^\pi\!\rmd \theta \, \frac{y_2(y_2 - x_2\cos\theta)}{\big[|x-y|^2 + 2x_2y_2(1-\cos\theta)\big]^{3/2}} \,,
\\ \label{H2}
H_2(x,y) & = \frac{1}{2\pi} \int_0^\pi\!\rmd \theta \, \frac{y_2 (x_1-y_1) \cos\theta}{\big[|x-y|^2 + 2x_2y_2(1-\cos\theta)\big]^{3/2}} \,.
\end{align}

We consider the time evolution of a single vortex in an external time-dependent field $F^\eps(x, t) = (F^\eps_1(x,t), F^\eps_2(x,t))$. The initial vorticity $\omega_\eps(x,0)$ is assumed of non-negative sign (the case of non-positive sign can be treated in the same manner) and satisfying conditions analogous to Eqs.~\eqref{initial}, \eqref{ai}, and \eqref{Mgamma}. Therefore, there are $M>0$, $a>0$, and $\zeta^0 = (z_0,r_0)$, with $r_0>0$, such that
\begin{equation}
\label{MgammaF}
0 \le \omega_\eps(x,0) \le \frac{M}{\eps^2|\log\eps|} \quad \forall\, x\in\bb R^2\,, \qquad |\log\eps|\int\!\rmd y\, \omega_\eps(y,0) =a\,,
\end{equation}
and 
\begin{equation}
\label{initialF}
\Lambda_\eps(0) := \supp\, \omega_\eps(\cdot,0) \subset \Sigma(\zeta^0|\eps)\,.
\end{equation}
The equation of motion is obtained replacing $u$ by $u+ F^\eps$ in  Eq.~\eqref{omeq}, i.e.,
\begin{equation}
\label{N-S_modified}
\partial_t \omega_\eps + \big[(u+ F^\eps)\cdot\nabla\big] \omega_\eps -\frac{u_2\omega_\eps}{x_2} = \nu \left[ \partial_{x_1}^2 \omega_\eps + \frac{1}{x_2}\partial_{x_2} (x_2\partial_{x_2} \omega_\eps) -\frac{\omega_\eps}{x_2^2} \right],
\end{equation}
where $\nabla=(\partial_{x_1}, \partial_{x_2})$. In the sequel, we shall refer to Eqs.~\eqref{N-S_modified} and \eqref{u=} as the ``reduced system''.

Concerning $F^\eps$, we suppose that it is the sum of three terms,
\begin{equation}
\label{ext_field}
F^\eps(x,t)= \widehat F^\eps(x,t) + \widetilde F^\eps(x,t) + \bar F^\eps(x,t) \,,
\end{equation}
each one being a smooth time-dependent field vanishing at infinity and satisfying the conditions 
\begin{align}
\label{divfree}
& \partial_{x_1}(x_2 \widehat F^\eps_1) +\partial_{x_2} (x_2 \widehat F^\eps_2) =0\,, \quad  \partial_{x_1}(x_2 \widetilde F^\eps_1) +\partial_{x_2} (x_2 \widetilde F^\eps_2) =0\,, \nonumber \\ & \partial_{x_1}(x_2 \bar F^\eps_1) +\partial_{x_2} (x_2 \bar F^\eps_2) =0\,,
\end{align}
which express the divergence free condition for the three-dimensional vector fields whose components in cylindrical coordinates are given by $(\widehat F^\eps_1, \widehat F^\eps_2, 0)$, $(\widetilde F^\eps_1, \widetilde F^\eps_2, 0)$, and $(\bar F^\eps_1, \bar F^\eps_2, 0)$, respectively. We further assume that

\smallskip
\noindent
(i) there is $\widehat r >0$ such that
\begin{equation}
\label{sing_Fc}
\supp\, \widehat F^\eps(\cdot,t) \subset \{x\in \bb R^2\colon  |x_2 - r_0| > \widehat r \} \quad \forall\, t\ge 0\,,
\end{equation}
and there is $\widehat C>0$ such that, any $(x,t)\in \bb R^2\times [0,+\infty)$,
\begin{equation}
\label{sing_F}
|\widehat F^\eps(x,t)| \le \frac{\widehat C}{\eps |\log\eps|}\,;
\end{equation}
without loss of generality, for later convenience, we also suppose
\begin{equation}
\label{rh<r_0/4}
\widehat r < \frac{r_0}4\,;
\end{equation}

\smallskip
\noindent
(ii) there is $\widetilde C>0$ such that, for any $(x,y,t)\in \bb R^2\times \bb R^2 \times [0,+\infty)$,
\begin{equation}
\label{reg_F}
|\widetilde F^\eps(x,t)| \le \frac{\widetilde C}{|\log\eps|}, \quad
|\widetilde F^\eps(x,t)-\widetilde F^\eps(y,t)| \le \frac{\widetilde C}{|\log\eps|} |x-y|\,;
\end{equation}

\smallskip
\noindent
(iii) there are $\bar C>0$ and $\beta>1$ such that, for any $(x,t)\in \bb R^2\times [0,+\infty)$,
\begin{equation}
\label{picc_F}
|\bar F^\eps(x,t)|\le \bar C \eps^\beta \quad \forall\, (x,t)\in \bb R^2\times [0,+\infty)\,.
\end{equation}

\smallskip
The above assumptions on $F^\eps$ are dictated by the need to apply the results of this section to the original problem of $N$ vortexes, in which $F^\eps$ simulates the action on each vortex by the other ones. More precisely, the reduced system is constructed to model the motion of a single vortex ring as long as the vorticity mass of each ring remains concentrated (as in the claim of Theorem \ref{thm:1}). Clearly, in principle, this can be true only for a very short initial time interval. But from the analysis of the reduced system we deduce stronger estimates on the localization property of the vortex ring, from which the proof of Theorem \ref{thm:1} can be achieved by means of a bootstrap argument. 

To be concrete, let us suppose that the reduced system describes the evolution of the $i$-th vortex ring. Therefore, by Eqs.~\eqref{omeq-i} and \eqref{u=},
\begin{equation}
\label{Fconj}
F^\eps(x,t) = F^{\eps,i}(x,t) := \sum_{j\ne i} \int\!\rmd y\, H(x,y)\, \omega_{j,\eps}(y,t)\,,
\end{equation}
and we further assume that the localization property stated in Theorem \ref{thm:1} holds true. Roughly speaking, the three terms in the right-hand side of Eq.~\eqref{ext_field} can then be interpreted in the following way. The term $\widetilde F^\eps(x,t)$ represents the sum of the velocity fields due to the vorticities $\omega_{j,\eps}(y,t)$ contained in a thin strip $|y_2-r_j| \ll 1$, when the argument $x$ of $\widetilde F^\eps(x,t)$ is outside this strip. The term $\bar F^\eps(x,t)$ represents the velocity field due to the vorticities $\omega_{j,\eps}(y,t)$ outside these thin strips; in this region, the integral of each $\omega_{j,\eps}(y,t)$ is very small and consequently the velocity very weak (this will be proved rigorously later on). Finally, $\widehat F^\eps(x,t)$ arises since $\omega_{i,\eps}(x,t)$ has not compact support, hence its support can invade regions around $r_j$, where the other rings can produce a singular velocity field (in the limit $\eps\to 0$), see Eq.~\eqref{sing_F}. Anyway, the vorticity mass of $\omega_{i,\eps}(x,t)$ in such regions is very small, hence this fact weaken the singularity of $\widehat F^\eps(x,t)$ in the integrals containing the product $\widehat F^\eps(x,t) \omega_{i,\eps}(x,t)$.

It is worthwhile to notice that in the case of the Euler equations ($\nu=0$) treated in Ref.~\cite{ButCavMar}, only the smoother term $\widetilde F^\eps(x,t)$ is present  because the supports of all the vortex rings remain confined in disjoint strips (say, the support of $\omega_{i,\eps}(x,t)$ remains inside $|x_2-r_j|\ll 1$ for any $j=1,\ldots,N$).

We can now state the localization result on the reduced system.

\begin{theorem}
\label{thm:2}
Assume the initial datum $\omega_\eps(x,0)$ verifies Eqs.~\eqref{MgammaF} and \eqref{initialF}, and define
\begin{equation}
\label{free_m}
\zeta(t) :=  \zeta^0 + \frac{a}{4\pi r_0} \begin{pmatrix} 1 \\ 0 \end{pmatrix} t \,.
\end{equation}
Then, for any $T>0$ the following holds true. For any $\eps$ small enough and $\nu\le \eps^2|\log\eps|^\gamma$, with $\gamma\in (0,1)$,  there are $\zeta_\eps(t)\in \Pi$  and $\varrho_\eps>0$ such that, for any $t\in [0, T]$,
\begin{equation}
\label{massa}
\lim_{\eps \to 0}|\log\eps| \int_{\Sigma(\zeta_\eps(t)|\varrho_\eps)}\!\rmd z\, \omega_\eps(x,t) = a\,,
\end{equation}
with
\[
\lim_{\eps\to 0} \varrho_\eps = 0, \qquad \lim_{\eps\to 0} \zeta_\eps(t) = \zeta(t) \,,
\]
where $\omega_\eps(x,t)$ is the time evolution of $\omega_\eps(x,0)$ via Eq.~\eqref{N-S_modified}.
\end{theorem}

The proof of Theorem \ref{thm:2} is the content of the following section, here we briefly summarize the whole strategy. 

We fix $\chi>2$ and define
\begin{equation}
\label{ansatz_massa_vort}
T_\eps := \sup\left\{t\in [0,T] \colon \int_{|x_2-r_0|> \widehat r} \!\rmd x \, \omega_\eps(x,t) \le \eps^\chi \;\; \forall\, s\in [0,t] \right\},
\end{equation}
with $\widehat r$ as in Eq.~\eqref{sing_Fc}. We assume $\eps < \widehat r$, so that, in view of Eq.~\eqref{initialF}, $T_\eps>0$ (but clearly, up to now, $T_\eps$ could possibly vanish as $\eps\to 0$).

We analyze the dynamics during the time interval $[0,T_\eps]$ and, making use of an a priori estimate on a particular integral quantity (the axial moment of inertia) proved in Lemma \ref{lem:2}, we show that for any $\eps$ small enough the vorticity mass remains highly concentrated in a thin horizontal strip centered around $x_2=r_0$ (note that, by Eq.~\eqref{initialF} this is true at time $t=0$). This strengthens the condition appearing in the definition Eq.~\eqref{ansatz_massa_vort} of $T_\eps$, so that, by continuity, we conclude that not only $T_\eps$ does not vanish as $\eps\to 0$ but actually $T_\eps = T$, definitively for any $\eps$ small enough. This is the content of Lemma \ref{lem:3} and Corollary \ref{cor:1}.

Next, in Lemma \ref{lem:4}, we first extend to the present context  a ``concentration result'' obtained in Ref.~\cite{BruMar} in absence of external field, which states that large part of the vorticity remains confined in a disk whose size is infinitesimal as $\eps\to 0$. At this point the proof of Theorem \ref{thm:2} is almost concluded. It remains to characterize the motion of the center of vorticity in the $x_1$-direction, with constant speed $a/(4\pi r_0)$ when $\eps\to 0$.

\subsection*{A notation warning} In what follows, we shall denote by $C$ a generic positive constant independent of $\eps$ and $\nu$, whose numerical value may change from line to line and it may possibly depend on the various parameters $\zeta^0=(z_0,r_0)$, $M, \widehat r, \widehat C, \widetilde C$, and $\bar C$ appearing in the assumptions on initial datum and external field, as well as on the given time $T$. Furthermore, throughout the proof the constant $a$ in Eq.~\eqref{MgammaF} will be put equal $1$.

\section{Proof of Theorem \ref{thm:2}}
\label{sec:4}

The following weak formulation will be used, which is a direct generalization of Eq.~\eqref{weq},
\begin{equation}
\label{weqF}
\begin{split}
\frac{\rmd}{\rmd t} \int\!\rmd x\, \omega_\eps(x,t) f(x,t) = &  \int\!\rmd x\, \omega_\eps(x,t) \big[ (u+F^\eps) \cdot \nabla f + \partial_t f \big](x,t) \\  & + \nu \int\!\rmd x\, \omega_\eps(x,t) \left(\Delta f - \frac{1}{x_2}\partial_{x_2} f \right)(x,t)\,,
\end{split}
\end{equation}
where $\Delta = \partial^2_{x_1}+\partial^2_{x_2}$. We remark also the generalization of Eq.~\eqref{cons-omr},
\begin{equation}
\label{cons-omrF}
\left[\partial_t + \left(u+F^\eps - \frac{3\nu}{x_2} \begin{pmatrix} 0 \\ 1 \end{pmatrix} \right) \cdot \nabla \right] \left(\frac{\omega_\eps}{x_2}\right)  = \nu \Delta \left(\frac{\omega_\eps(x,t)}{x_2} \right),
\end{equation}
which implies, in particular, a parabolic maximum principle for the ratio $\omega_\eps/x_2$, see Ref.~\cite{UY},
\begin{equation}
\label{maxpri}
\sup_x\frac{\omega_\eps(x,t)}{x_2} \le  \sup_x\frac{\omega_\eps(x,0)}{x_2} \quad \forall\, t\ge 0\,.
\end{equation}
Finally, we notice that the total vorticity is non-increasing in time. More precisely, it can be proved that
\begin{equation}
\label{dotom<}
\frac{\rmd}{\rmd t} \int\! \rmd x \, \omega_\eps(x,t) = -2 \nu \int\! \rmd x_1 \, \partial_{x_2}\omega_\eps(x_1,0,t) < 0\,,
\end{equation}
see Ref.~\cite[Eq. (68)]{Gal13}. Therefore, from Eq.~\eqref{MgammaF} with $a=1$,
\begin{equation}
\label{cons-mass}
\int\!\rmd x\, \omega_\eps(x,t) \le \frac{1}{|\log\eps|} \quad \forall\, t\ge 0\,.
\end{equation}

Recalling Eq.~\eqref{u=}, it is useful to decompose the velocity field as
\begin{equation}
\label{decom_u}
u(x,t) = \widetilde u(x,t)  + \int\!\rmd y\, L(x,y)\, \omega_\eps(y,t) +  \int\!\rmd y\, \mc R(x,y) \omega_\eps(y,t) \,,
\end{equation}
where $\widetilde u(x,t)=\int\!\rmd y\, K(x-y)\, \omega_\eps(y,t)$ with\footnote{We adopt the notation $v^\perp :=(v_2,-v_1)$ for any given $v = (v_1,v_2)$.}
\begin{equation}
\label{vel-vor3}
K(x) = \nabla^\perp  \mc G(x)\,, \quad \mc G(x) := - \frac{1}{2\pi} \log|x|\,,
\end{equation}
and
\begin{equation}
\label{bc_bound}
L(x,y) = \frac{1}{4\pi x_2} \log\frac {1+|x-y|}{|x-y|}\begin{pmatrix} 1 \\ 0 \end{pmatrix}.
\end{equation}
In Ref.~\cite[Appendix B]{BuM2} it is shown that there is $C_0>0$ such that  the remaining term $\mc R(x,y) = (\mc R_1(x,y), \mc R_2(x,y)) = H(x,y) - K(x-y) - L(x,y)$ satisfies, for some $C_0>0$,
\begin{equation}
\label{sR}
|\mc R_1(x,y)| \le C_0 \frac{1+x_2+\sqrt{x_2y_2} \big(1+ |\log(x_2y_2)|\big)}{x_2^2}\,, \quad |\mc R_2(x,y)| \le \frac{C_0}{x_2}\,.
\end{equation}
It will also useful the following estimate, recently proved in Ref.~\cite{Fe-Sv}: if $u$ is  given by Eq.~\eqref{u=} then
\begin{equation}
\label{u_piccolo}
\| u \|_{L^\infty}\le C_* \| \omega_\eps \|_{L^1}^{1/4}
\| x_2^2 \,\omega_\eps \|_{L^1}^{1/4} \| \omega_\eps /x_2 \|_{L^\infty}^{1/2}\,,
\end{equation}
where $C_*>0$ is an absolute constant.

\subsection{Estimates on the radial motion}
\label{sec:4.1}

In the quoted papers Refs.~\cite{BuM2,ButCavMar}, a central role is played by the center of vorticity $B_\eps(t)= (B_{\eps,1}(t), B_{\eps,2}(t))$ and the axial moment of inertia with respect to $B_{\eps,2}(t)$, defined by
\begin{align}
\label{c.m.}
B_{\eps}(t) & = |\log\eps|  \int\!\rmd x\, \omega_\eps(x,t) \, x\,, \\ \label{moment} I_\eps(t) & = \int\! \rmd x\,  \omega_\eps(x,t) \left(x_2-B_{\eps,2}(t)\right)^2\,,
\end{align}
whose time evolution along the Euler flow can be efficiently controlled, giving rise, in particular, to an upper bound on $I_\eps(t)$. It is worthwhile to remark that in our case  $B_\eps(t)$ is not really equal to the center of vorticity because, due to Eq.~\eqref{dotom<}, $\int\!\rmd x\, \omega_\eps(x,t) < |\log\eps|^{-1}$ for $t>0$. On the other hand, as a straightforward byproduct of our analysis, we will obtain that
\[
\lim_{\eps\to 0} \left| B_{\eps}(t) - \frac{\int\!\rmd x\, \omega_\eps(x,t) \,x}{\int\!\rmd x\, \omega_\eps(x,t)} \right| = 0 \quad \forall\,t\in [0,T]\,,
\]
and for this reason $B_\eps(t)$ will be (improperly) named center of vorticity also in this manuscript.

Anyway, this analysis would require the application of Eq.~\eqref{weqF} to compute the time derivative of $B_\eps(t)$, which presents a difficulty. Indeed, when $\nu\ne 0$ the vorticity does not remain supported away from the symmetry axis, and the last term in Eq.~\eqref{weqF}, as well as the kernels $L(x,y)$ and $\mc R(x,y)$ appearing in Eq.~\eqref{decom_u}, are singular as $x_2 \to 0$. This unessential singularity can be avoided by introducing the following cut-offed versions of $B_\eps(t)$ and $I_\eps(t)$,
\begin{align}
\label{c.m.*}
B_\eps^*(t) & = |\log\eps|  \int\!\rmd x\, \omega_\eps(x,t)\, G(x_2) \, x \,, \\ \label{moment*} I_\eps^*(t) & =  \int\!\rmd x\, \omega_\eps(x,t) \left(x_2-B_{\eps,2}^*(t)\right)^2G(x_2) \,,
\end{align}
where
\begin{equation}
\label{G}
G\in C^\infty(\bb R;[0,1])  \quad \text{ is such that }\,\, G(x_2) = 
\begin{cases}
1 &\text{if }\;\; x_2\ge r_0-\widehat r\,, \\ 0 &\text{if }\;\; x_2 \le (r_0-\widehat r)/2\,.
\end{cases}
\end{equation}
We analyze the time evolution of  $B_{\eps,2}^*(t)$ and $I_\eps^*(t)$ up to time $T_\eps$, showing that $B_{\eps,2}^*(t)$ and $I_\eps^*(t)$ enjoy the same kind of estimates obtained in Ref.~\cite{ButCavMar} for $B_{\eps,2}(t)$ and $I_\eps(t)$ in the Euler case.\footnote{Notice however that, by Eq.~\eqref{ansatz_massa_vort},
\[
|B_{\eps,2}^*(t) -B_{\eps,2}(t) | \le C |\log\eps|\eps^\chi \,, \quad |I_\eps^*(t) - I_\eps(t)| \le C \eps^\chi \qquad \forall\, t\in [0,T_\eps]\,.
\]
}

\begin{lemma}
\label{lem:2}
Under the hypothesis of Theorem \ref{thm:2} we have that
\begin{equation}
\label{growth B}
|\dot B_{\eps,2}^*(t)| \leq \frac{C}{|\log\eps|} \quad \forall\, t\in [0,T_\eps]
\end{equation}
and
\begin{equation}
\label{Iee}
I_\eps^*(t) \le \frac{C}{|\log\eps|^2} \quad \forall\, t\in [0,T_\eps]\,.
\end{equation}
\end{lemma}

\begin{proof}
Using the decomposition Eq.~\eqref{decom_u} for $u(x,t)$, observing that 
\begin{equation}
\label{tu=}
\int\!\rmd x\, \widetilde u(x,t)\,\omega_\eps(x,t) = 0
\end{equation}
(which comes from the explicit form of $K(x)$ given in Eq.~\eqref{vel-vor3}), and since $L_2(x,y)=0$, see Eq.~\eqref{bc_bound}, from Eq.~\eqref{weqF} we get
\[
\begin{split}
& \dot B_{\eps,2}^*(t) = |\log\eps|\int\! \rmd x\,\omega_\eps(x,t)\,  \widetilde u_2(x,t)\,  [G(x_2) + x_2 G'(x_2)-1]\\ & \qquad+ |\log\eps|\int\! \rmd x\,\omega_\eps(x,t)\, \bigg(\int\!\rmd y\, \mc R_2(x,y) \omega_\eps(y,t) + F^\eps_2(x,t) \bigg) [G(x_2)+ x_2 G'(x_2)]    \\ &\qquad  + \nu |\log\eps|\int\! \rmd x\,\omega_\eps(x,t)\, \bigg(x_2G''(x_2) + G'(x_2) - \frac{G(x_2)}{x_2}\bigg) =: T_1 + T_2 + T_3\,.
\end{split}
\]
From Eq.~\eqref{G} and \eqref{vel-vor3},
\[
\begin{split}
|T_1| & \le C |\log\eps| \int_{x_2 \le r_0 - \widehat r} \!\rmd x \,  \omega_\eps(x,t) |\widetilde u_2(x,t)| \\ & \le C |\log\eps| \int_{x_2 \le r_0 - \widehat r} \!\rmd x \, \omega_\eps(x,t) \int\! \rmd y\, \frac{\omega_\eps(y,t)}{2\pi|x-y|} \\ & \le C |\log\eps| \int_{x_2 \le r_0 - \widehat r} \!\rmd x \, \omega_\eps(x,t) \int_{y_2 > r_0 + \widehat r} \! \rmd y\, \frac{\omega_\eps(y,t)}{2\pi|x-y|} \\ & \quad +C |\log\eps| \int_{x_2 \le r_0 - \widehat r} \!\rmd x \, \omega_\eps(x,t) \int_{y_2\le  r_0 + \widehat r} \! \rmd y\, \frac{\omega_\eps(y,t)}{2\pi|x-y|}  =: T_{1,1} + T_{1,2}\,.
\end{split}
\]
Clearly, by Eq.~\eqref{ansatz_massa_vort},
\[
T_{1,1} \le C |\log\eps| \int_{x_2 \le r_0 - \widehat r} \!\rmd x \, \omega_\eps(x,t) \int_{y_2 > r_0 + \widehat r} \! \rmd y\, \frac{\omega_\eps(y,t)}{4\pi\widehat r} \le C |\log\eps| \eps^{2\chi} \quad \forall\, t\in [0,T_\eps]\,.
\]
To estimate $T_{1,2}$, we observe that since $1/|x-y|$ is monotonically unbounded as $y\to x$, the maximum of $\int_{y_2\le  r_0 + \widehat r} \rmd y\, \omega_\eps(y,t)/|x-y|$ is achieved when we rearrange the vorticity mass as close as possible to the singularity. On the other hand, from the maximum principle Eq.~\eqref{maxpri} and the assumptions Eqs.~\eqref{MgammaF} and \eqref{initialF}, we have
\[
\omega(y,t) \le \frac{r_0 + \widehat r}{r_0-\eps} \frac{M}{\eps^2|\log\eps|} =: \mc M_\eps \quad \forall\, y_2\in [0,r_0+\widehat r]\,,
\]
and, by Eq.~\eqref{cons-mass}, $\int_{y_2\le  r_0 + \widehat r} \rmd y\, \omega_\eps(y,t) \le |\log\eps|^{-1}$, whence
\[
\int_{y_2\le  r_0 + \widehat r} \! \rmd y\, \frac{\omega_\eps(y,t)}{2\pi|x-y|} \le  \mc M_\eps \int_{\Sigma (0|\rho_\eps)}\!\rmd y'\, \frac{1}{2\pi|y'|} = \mc M_\eps \rho_\eps \;, 
\]
where the radius $\rho_\eps$ is such that $\mc M_\eps \pi \rho_\eps^2 = |\log\eps|^{-1}$. Therefore, by Eq.~\eqref{ansatz_massa_vort},
\[
T_{1,2} \le C \eps^{\chi-1} \quad \forall\, t\in [0,T_\eps]\,.
\]
Regarding the term $T_2$, we have
\[
\begin{split}
|T_2| & \le C |\log\eps| \int_{x_2 \ge (r_0 - \widehat r)/2} \!\rmd x \,  \omega_\eps(x,t) \bigg(\int\!\rmd y\, |\mc R_2(x,y)|  \omega_\eps(y,t) +  |F^\eps_2(x,t)| \bigg) \\ & \le  C |\log\eps| \bigg\{\int_{x_2 \ge (r_0 - \widehat r)/2} \!\rmd x \, \omega_\eps(x,t) \frac{2C_0}{r_0 - \widehat r} \int\!\rmd y\,\omega_\eps(y,t) \\ & \quad + \frac{\widehat C}{\eps|\log\eps|} \int_{|x_2 - r_0| > \widehat r} \!\rmd x \,  \omega_\eps(x,t) + \bigg(\frac{\widetilde C}{|\log\eps|} + \bar C \eps^\beta\bigg) \int  \!\rmd x \,\omega_\eps(x,t) \bigg\} \\ & \le C \bigg(\eps^{\chi-1} + \frac{1}{|\log\eps|} + \eps^\beta\bigg) \qquad \forall\, t\in [0,T_\eps]\,,
\end{split}
\]
where we used Eqs.~\eqref{G}, \eqref{sR}, \eqref{sing_Fc}, \eqref{sing_F}, \eqref{reg_F}, \eqref{picc_F}, \eqref{ansatz_massa_vort}, and \eqref{cons-mass}. Finally, recalling we assume $\nu\le \eps^2|\log\eps|^\gamma$ and using again Eq.~\eqref{cons-mass},
\[
|T_3| \le C \nu |\log\eps| \frac{2}{r_0-\widehat r} \int_{x_2 \ge (r_0 - \widehat r)/2} \!\rmd x \,  \omega_\eps(x,t) \le C \eps^2|\log\eps|^\gamma\,.
\]
Recalling $\chi>2$, Eq.~\eqref{growth B} now follows from the above estimates.

To prove the second estimate Eq.~\eqref{Iee} we need to compute the time derivative $\dot I_\eps^*(t)$. From Eq.~\eqref{weqF} we have,
\[
\dot I_\eps^*(t)  = \mc K_1 +\mc K_{2,1} + \mc K_{2,2} + \mc K_3 + \mc K_4 \,,
\]
where
\[
\begin{split}
\mc K_1 & = \int\! \rmd x\,\omega_\eps(x,t)\, u_2(x,t) \big(x_2-B_{\eps,2}^*(t)) [2G(x_2) + (x_2-B_{\eps,2}^*(t)) G'(x_2)]\,, \\  \mc K_{2,1} & = \int\! \rmd x\,\omega_\eps(x,t)\, F^\eps_2(x,t) \big(x_2-B_{\eps,2}^*(t))^2 G'(x_2) \,, \\ \mc K_{2,2} & =  \int\! \rmd x\,\omega_\eps(x,t)\, F^\eps_2(x,t) \big(x_2-B_{\eps,2}^*(t)) 2G(x_2)\,, \\ \mc K_3 & = \nu \int\! \rmd x\,\omega_\eps(x,t)\, \bigg[ \frac{2B_{\eps,2}^*(t)}{x_2} G(x_2) + 4(x_2-B_{\eps,2}^*(t))  G'(x_2) \\ & \qquad\quad - (x_2-B_{\eps,2}^*(t))^2 \frac{G'(x_2)}{x_2} + (x_2-B_{\eps,2}^*(t))^2 G''(x_2)  \bigg], \\ \mc K_4 & = 2 \dot B_{\eps,2}^*(t)  \int\! \rmd x\,\omega_\eps(x,t)\, (B_{\eps,2}^*(t) - x_2) G(x_2)\,.
\end{split}
\]
We now remind that 
\begin{equation}
\label{M2}
M_2(t) := \int\!\rmd x\, \omega(x,t) x_2^2
\end{equation}
is a conserved quantity in absence of external field, see, e.g., Refs.~\cite{BCM00,Gal13}. More precisely, noticing that the right-hand side of Eq.~\eqref{weqF} vanishes for $f(x,t) = x_2^2$, this means that
\begin{equation}
\label{M2d=0}
\int\!\rmd x\, \omega_\eps(x,t) u_2(x,t) x_2 = 0\,,
\end{equation}
which can also be verified by direct inspection using Eqs.~\eqref{u=} and \eqref{H2}. This implies, together with Eq.~\eqref{tu=} and using Eq.~\eqref{decom_u}, that $\mc K_1$ can be rewritten as
\[
\begin{split}
\mc K_1 & = \int\! \rmd x\,\omega_\eps(x,t)\, \widetilde u_2(x,t) (x_2-B_{\eps,2}^*(t)) [2G(x_2) - 2  + (x_2-B_{\eps,2}^*(t)) G'(x_2)] \\ & \quad +  \int\! \rmd x\,\omega_\eps(x,t)\int\!\rmd y\,\mc R_2(x,y) \omega_\eps(y,t)  x_2 [2G(x_2) - 2  + (x_2-B_{\eps,2}^*(t)) G'(x_2)] \\ & \quad - B_{\eps,2}^*(t) \int\! \rmd x\,\omega_\eps(x,t) \int\!\rmd y\,\mc R_2(x,y) \omega_\eps(y,t) [2G(x_2) + (x_2-B_{\eps,2}^*(t)) G'(x_2)]\,.
\end{split}
\]
Note also that Eqs.~\eqref{initialF} and \eqref{growth B} give
\begin{equation}
\label{Br0}
|B_{\eps,2}^*(t) - r_0| \le \frac{C}{|\log\eps|} \quad \forall\, t\in [0,T_\eps]\,, 
\end{equation} 
and, in particular,
\begin{equation}
\label{boundB}
B_{\eps,2}^*(t) \le C \quad \forall\, t\in [0,T_\eps]\,.
\end{equation}
Therefore, because of the definition Eq.~\eqref{G} of $G$,
\[
\begin{split}
|\mc K_1| & \le C \int_{x_2 \le r_0-\widehat r} \! \rmd x\,\omega_\eps(x,t)\, \bigg( |\widetilde u_2(x,t)| +  x_2 \int\!\rmd y\,|\mc R_2(x,y)|  \omega_\eps(y,t) \bigg) \\ & \quad + C \int_{x_2 \ge (r_0-\widehat r)/2} \! \rmd x\,\omega_\eps(x,t) \int\!\rmd y\,|\mc R_2(x,y)| \omega_\eps(y,t)  \\ & \le C \int_{x_2 \le r_0-\widehat r} \! \rmd x\,\omega_\eps(x,t)\, \bigg( |\widetilde u_2(x,t)| + C_0 \int\!\rmd y\, \omega_\eps(y,t) \bigg) \\ & \quad + C \int_{x_2 \ge (r_0-\widehat r)/2} \! \rmd x\,\omega_\eps(x,t)  \frac{2C_0}{r_0 - \widehat r} \int\!\rmd y\,\omega_\eps(y,t)  \\ & \le C \bigg(\eps^{\chi-1} + \frac{1}{|\log\eps|^2}\bigg) \quad \forall\, t\in [0,T_\eps]\,,
\end{split}
\]
where in the last inequality we applied the previous estimate on $T_1$ and Eq.~\eqref{cons-mass}. Analogously,
\[
\begin{split}
|\mc K_{2,1}| & \le C  \int_{x_2 \ge (r_0 - \widehat r)/2} \!\rmd x \,  \omega_\eps(x,t) |F^\eps_2(x,t)| \\ & \le \frac{C}{|\log\eps|} \bigg(\eps^{\chi-1} + \frac{1}{|\log\eps|} + \eps^\beta\bigg) \quad \forall\, t\in [0,T_\eps]
\end{split}
\]
(see the previous estimate on $T_2$) and, from the Cauchy-Schwarz inequality,
\[
\begin{split}
|\mc K_{2,2}| & \le \bigg(\int\! \rmd x\,\omega_\eps(x,t)G(x_2) F^\eps_2(x,t)^2 \bigg)^{1/2}  \bigg(\int\! \rmd x\,\omega_\eps(x,t) (B_{\eps,2}^*(t) - x_2)^2 G(x_2)\bigg)^{1/2} \\ & \le  \frac{C}{|\log\eps|^{3/2}} \sqrt{\eps^{\chi-2}|\log\eps| + 1 + \eps^{2\beta} |\log\eps|^2}\, \sqrt{I_\eps^*(t)}  \quad \forall\, t\in [0,T_\eps]
\end{split}
\]
(see again the previous estimate on $T_2$ with now $ F^\eps_2(x,t)^2$ in place of $| F^\eps_2(x,t)|$). Finally,
\[
|\mc K_3| \le  C \nu \int_{x_2 \ge (r_0 - \widehat r)/2} \!\rmd x \,  \omega_\eps(x,t) \le C \eps^2|\log\eps|^{\gamma-1}\,
\]
(see the previous estimate on $T_3$) and, from the Cauchy-Schwarz inequality and Eq.~\eqref{growth B},
\[
\begin{split}
|\mc K_4| & \le \frac{C}{|\log\eps|}  \bigg(\int\! \rmd x\,\omega_\eps(x,t)G(x_2)\bigg)^{1/2}  \bigg(\int\! \rmd x\,\omega_\eps(x,t) (B_{\eps,2}^*(t) - x_2)^2 G(x_2)\bigg)^{1/2} \\ & \le  \frac{C}{|\log\eps|^{3/2}}\sqrt{I_\eps^*(t)}  \quad \forall\, t\in [0,T_\eps]\,,
\end{split}
\]
where in the last inequality we used  Eq.~\eqref{cons-mass} and $G\le 1$.

Collecting together all the previous estimates and recalling $\chi>2$ we conclude that 
\[
|\dot I_\eps^*(t)| \le \frac{C}{|\log\eps|^{3/2}}\sqrt{I_\eps^*(t)} + \frac{C}{|\log\eps|^2}  \quad \forall\, t\in [0,T_\eps]\,.
\]
Since Eq.~\eqref{initialF} implies $I_\eps^*(0) \leq 4\eps^2$, Eq.~\eqref{Iee} follows from the last differential inequality. Indeed, given $M>0$ let 
\[
T_M = \sup\left\{ t\in [0,T_\eps]\colon I_\eps^*(s) \le \frac{M}{|\log\eps|^2} \;\;  \forall\, s\in [0,t] \right\}.
\]
Clearly, $T_M>0$ for any $\eps$ small enough and
\[
|\dot I_\eps^*(t)| \le \frac{C\sqrt M}{|\log\eps|^{5/2}} + \frac{C}{|\log\eps|^2} \quad \forall\, t\in [0,T_M]\,.
\]
This implicates, as $T_M \le T_\eps \le T$,
\[
I_\eps^*(t) \le 4\eps^2 + \frac{CT\sqrt M}{|\log\eps|^{5/2}} + \frac{CT}{|\log\eps|^2} \quad \forall\, t\in [0,T_M]\,.
\]
If $M$ is chosen large enough then, for any $\eps$ small enough,
\[
I_\eps^*(t) \le \frac{M}{2|\log\eps|^2}  \quad \forall\, t\in [0,T_M]\,,
\]
which implies $T_M=T_\eps$ by continuity, whence Eq.~\eqref{Iee} (with $C=M$).
\end{proof}

As already mentioned, our aim is to strengthen the condition appearing in the definition Eq.~\eqref{ansatz_massa_vort}. This is the content of the next lemma, where we show that the vorticity mass out of a thin strip parallel to the $x_1$-axis and centered around $x_2=r_0$ is very small (less than any power of $\eps$, for $\eps\to 0$). With respect to  the analogous Ref.~\cite[Lemma 3.4]{ButCavMar}, here we have to take account of the presence of the viscosity (in particular, this implies that the vorticity is not compactly supported) and the presence of a less regular external field. However, these differences do not condition too much the reasoning and we just sketch the proof in Appendix \ref{sec:A}, where we explain in more detail the new parts.

\begin{lemma} 
\label{lem:3}
Let $m_t$ be defined as 
\begin{equation}
\label{mt}
m_t(h) = \int_{|x_2-B_{\eps,2}^*(t)|>h}\!\rmd x\,\omega_\eps(x,t) \,.
\end{equation}
Then, under the hypothesis of Theorem \ref{thm:2}, for each $\ell>0$  and  $k \in \big(0, \frac14\big)$,
\begin{equation}
\label{smt}
\lim_{\eps\to 0} \, \max_{ t \in [0, T_\eps]} \eps^{-\ell} m_t \left(\frac{1}{|\log\eps|^k} \right)  = 0\,.
\end{equation}
\end{lemma}

\begin{corollary}
\label{cor:1}
Under the hypothesis of Theorem \ref{thm:2}, there is $\eps_0\in (0,1)$ such that $T_\eps=T$ for any $\eps\in (0,\eps_0)$ and, for each $\ell>0$ and $k \in \big(0, \frac14\big)$,
\begin{equation}
\label{smt*}
\lim_{\eps\to 0} \, \max_{t \in [0, T]} \eps^{-\ell} \int_{|x_2-r_0|>\frac{1}{|\log\eps|^k}}\!\rmd x\,\omega_\eps(x,t) = 0\,.
\end{equation}
\end{corollary}

\begin{proof}
From Eqs.~\eqref{Br0} and \eqref{smt} we deduce that for each $\ell>0$ and $k\in (0,\frac 14)$,
\begin{equation}
\label{smt1}
\lim_{\eps\to 0} \, \max_{t \in [0, T_\eps]} \eps^{-\ell} \int_{|x_2-r_0|>\frac{1}{|\log\eps|^k}}\!\rmd x\,\omega_\eps(x,t) = 0\,.
\end{equation}
Choosing, e.g., $\ell>2\chi$, where $\chi>2$ is the parameter appearing in Eq.~\eqref{ansatz_massa_vort}, it follows that there is $\eps_0\in (0,1)$ such that 
\[
\int_{|x_2-r_0|>\widehat r}\!\rmd x\,\omega_\eps(x,t) \le C \eps^\ell \le \eps^{2\chi} \quad \forall\,t \in [0,T_\eps] \quad \forall\, \eps\in (0,\eps_0).
\]
By continuity $T_\eps =T$ for any $\eps\in (0,\eps_0)$. Clearly, Eq.~\eqref{smt*} then follows from Eq.~\eqref{smt1}.
\end{proof}

\subsection{Analysis of the axial motion}
\label{sec:4.2}

The first step is a concentration result which shows that large part of the vorticity mass remains confined in a disk whose size is infinitesimal as $\eps\to 0$.

\begin{lemma}
\label{lem:4}
Assume the hypothesis of Theorem \ref{thm:2} and let $\eps_0$ as in Corollary \ref{cor:1}. For any $\eta \in (\gamma,1)$ there are $\eps_1\in (0,\eps_0)$, $\varrho_1>0$, $C_1>0$, and $q_\eps(t)\in \bb R^2$, such that
\begin{equation}
\label{lem1b}
|\log\eps| \int_{\Sigma(q_\eps(t),\varrho_\eps)}\!\rmd x\, \omega_\eps(x,t) \ge 1 - \frac{C_1}{|\log\eps|^{\eta-\gamma}} \qquad \forall\, t\in [0,T] \quad  \forall\, \eps\in (0,\eps_1)\,,
\end{equation}
with $\varrho_\eps = \varrho_1\eps \exp(|\log\eps|^\eta)$.
\end{lemma}
For later purpose, we remark that, by Eqs.~\eqref{MgammaF} (with $a=1$) and \eqref{dotom<},  Eq.~\eqref{lem1b} implies
\begin{equation}
\label{omega-q-con}
\lim_{\eps\to 0} |\log\eps| \int\!\rmd x\, \omega_\eps(x,t) = 1 \quad \forall\, t\in [0,T]\,,
\end{equation}
i.e., the vorticity mass is almost conserved for $\eps$ small.

In absence of the external field, the content of Lemma \ref{lem:4} has been proved in Ref.~\cite{BruMar}, which extends a similar result for the Euler case treated in Ref.~\cite{BCM00}. The Euler case in presence of a regular external field (i.e., satisfying Eq.~\eqref{reg_F}) has been analyzed in Ref.~\cite[Lemma 4.1]{ButCavMar}. The proof of Lemma \ref{lem:4} is given in Appendix \ref{sec:B}.

\medskip
From Lemma \ref{lem:4}, Eq.~\eqref{massa} holds with $\zeta_\eps(t)=q_\eps(t)$ so that, recalling Eq.~\eqref{free_m} and that throughout this section we set $a=1$, the proof of Theorem \ref{thm:2} is completed if we show that
\begin{equation}
\label{qq}
\lim_{\eps\to 0} q_{\eps,2}(t) = r_0\,, \quad  \lim_{\eps\to 0} q_{\eps,1}(t) = z_0 + \frac{t}{4\pi r_0} \quad \forall\, t\in [0,T]\,.
\end{equation}

The first limit in Eq.~\eqref{qq} is an immediate consequence of Eqs.~\eqref{smt*} and \eqref{omega-q-con}. Indeed, by absurd, if there were a sequence $\eps_j \to 0$ such that $\varliminf_{j\to \infty} |q_{\eps_j,2}(t)-r_0| >0$ then the strip $S_j := \{x\colon |x_2-r_0|\le |\log\eps_j|^{-k} \}$ and the disk $\Sigma_j := \Sigma(q_{\eps_j}(t),\varrho_{\eps_j})$ would be disjoint for any large enough $j$, while the vorticity should concentrate in both of them, i.e.,
\[
\lim_{j \to \infty} |\log\eps_j| \int_{S_j}\!\rmd x\, \omega_{\eps_j}(x,t)
= \lim_{j \to \infty}  |\log\eps_j| \int_{\Sigma_j}\!\rmd x\, \omega_{\eps_j}(x,t) = 1\,.
\]

As explained below, the second limit in Eq.~\eqref{qq} can be deduced from the following proposition, whose proof is given at the end of the section.

\begin{proposition}
\label{prop:1}
Assume the hypothesis of Theorem \ref{thm:2}. Let $B_{\eps,1}(t)$ be the first component of the center of vorticity Eq.~\eqref{c.m.} and let $J_\eps(t)$ be the corresponding radial moment of inertia,
\begin{equation}
\label{moment1} 
J_\eps(t) = \int\! \rmd x\,  \omega_\eps(x,t) \left(x_1-B_{\eps,1}(t)\right)^2 \,.
\end{equation}
Then,
\begin{align}
\label{b1c}
\lim_{\eps\to 0} B_{\eps,1}(t) & = z_0 + \frac{t}{4\pi r_0} \quad \forall\, t\in [0,T]\,,
\\ \label{Jst} J_\eps(t) & \le \frac{C}{|\log\eps|} \quad \forall\, t\in [0,T]\,.
\end{align}
\end{proposition}

Indeed, we claim that Eq.~\eqref{Jst} implies
\begin{equation}
\label{q1b1}
\lim_{\eps\to 0} |B_{\eps,1}(t) - q_{\eps,1}(t)| = 0  \quad \forall\, t\in [0,T]\,,
\end{equation}
which, together with Eq.~\eqref{b1c}, gives the second limit in Eq.~\eqref{qq}. To prove Eq.~\eqref{q1b1}, we first notice that
\[
\begin{split}
B_{\eps,1}(t) - q_{\eps,1}(t) & = \left(1- \frac{1}{|\log\eps|\int\!\rmd x\, \omega_\eps(x,t)} \right) B_{\eps,1}(t) \\ & \quad + \frac{1}{|\log\eps|\int\!\rmd x\, \omega_\eps(x,t)} |\log\eps| \int\! \rmd x\, \omega_\eps(x,t) [x_1 - q_{\eps,1}(t)]\,.
\end{split}
\]
By Eq.~\eqref{b1c}, $|B_{\eps,1}(t)|$ is bounded uniformly for any $\eps$ small and $t\in [0,T]$, therefore Eq.~\eqref{omega-q-con} implies it is sufficient to show that
\begin{equation}
\label{q1b1bis}
\lim_{\eps\to 0}\mc P_\eps(t) = 0  \quad \forall\, t\in [0,T]\,, \qquad \text{where }\;\; \mc P_\eps(t) := |\log\eps| \int\! \rmd x\, \omega_\eps(x,t) |x_1 - q_{\eps,1}(t)|\,.
\end{equation}
To this end, we set $\Sigma_t = \Sigma(q_\eps(t), \varrho_\eps)$ and estimate,
\[
\begin{split}
\mc P_\eps(t)  & = |\log\eps|\int_{\Sigma_t}\! \rmd x\, |x_1 - q_{\eps,1}(t)| \omega_\eps(x,t) +  |\log\eps|\int_{\Sigma_t^\complement}\! \rmd x\,  |x_1 - q_{\eps,1}(t)| \omega_\eps(x,t) \\ & \le \varrho_\eps + \frac{C_1}{|\log\eps|^{\eta-\gamma}} |B_{\eps,1}(t) - q_{\eps,1}(t)| +  |\log\eps|\int_{\Sigma_t^\complement}\! \rmd x\, |x_1-B_{\eps,1}(t)| \omega_\eps(x,t)\,,
\end{split}
\]
where we used Eq.~\eqref{lem1b}. Therefore, by assuming $\eps$ so small to have $2C_1 \le |\log\eps|^{\eta-\gamma}$,
\[
\begin{split}
\mc P_\eps(t)  & \le 2 \varrho_\eps + 2   \sqrt{|\log\eps|\int_{\Sigma_t^\complement}\! \rmd x\, \omega_\eps(x,t)}  \sqrt{|\log\eps|\int_{\Sigma_t^\complement}\! \rmd x\, \left(x_1-B_{\eps,1}(t)\right)^2 \omega_\eps(x,t)} \\ & \le 2 \varrho_\eps + 2  \sqrt{\frac{C_1}{|\log\eps|^{\eta-\gamma}}} \sqrt{ |\log\eps| J_\eps(t)}\,,
\end{split}
\]
where we applied again Eq.~\eqref{lem1b}. This estimate and Eq.~\eqref{Jst} clearly imply Eq.~\eqref{q1b1bis}. 

As a final remark, we notice that (thanks to the control on the moment of inertia) we have identified the points $q_\eps(t)$ and $B_\eps(t)$ in the limit $\eps\to 0$, while in Ref.~\cite{BCM00,BruMar} only the properties Eq.~\eqref{qq} were stated. 

\begin{proof}[Proof of Proposition \ref{prop:1}] 
We start with the proof of Eq.~\eqref{b1c}.  By Eq.~\eqref{weqF} with $f(x,t) =x_1$ we have,
\begin{align}
\label{bpunto}
\dot B_{\eps,1}(t) & = |\log\eps| \int\!\rmd x\, \omega_\eps(x,t) \, (u_1+F^\eps_1)(x,t) \nonumber \\ & = |\log\eps| \int_{|x_2 - r_0| > \widehat r} \!\rmd x\, \omega_\eps(x,t) \, \left(u_1+F^\eps_1 - \widetilde u_1\right)(x,t) \nonumber \\ & \quad + |\log\eps|\int_{|x_2 - r_0| \le \widehat r} \!\rmd x\, \omega_\eps(x,t) \left( \int\!\rmd y\, \mc R_1(x,y) \omega_\eps(y,t) + F^\eps_1 (x,t)\right) \nonumber \\ & \quad + |\log\eps|\int_{|x_2 - r_0| \le \widehat r} \!\rmd x\, \omega_\eps(x,t) \int\!\rmd y\, L_1(x,y) \omega_\eps(y,t) \,,
\end{align}
where we used Eq.~\eqref{decom_u} and (the first component of) \eqref{tu=}. From Eqs.~\eqref{sing_F}, \eqref{reg_F}, \eqref{picc_F}, and \eqref{uinf<},
\[
\left||\log\eps| \int_{|x_2 - r_0| > \widehat r} \!\rmd x\, \omega_\eps(x,t) \, \left(u_1 + F^\eps_1 \right)(x,t)\right| \le \frac{C}{\eps|\log\eps|} \int_{|x_2 - r_0| > \widehat r} \!\rmd x\, \omega_\eps(x,t)\,,
\]
and the right-hand side vanishes as $\eps\to 0$ by Eq.~\eqref{smt*}. Concerning the term containing $\widetilde u_1$, we observe that
\[
|\widetilde u_1 (x,t)| \le  \int_{|y_2 - x_2| \le \widehat r} \! \rmd y\, \frac{\omega_\eps(y,t)}{2\pi|x-y|} + \int_{|y_2 - x_2| > \widehat r} \! \rmd y\, \frac{\omega_\eps(y,t)}{2\pi|x-y|}\,.
\]
An upper bound of the first integral in the right-hand side can be obtained by rearrangement. More precisely, from the maximum principle Eq.~\eqref{maxpri} and the assumptions Eqs.~\eqref{MgammaF} and \eqref{initialF},
\[
\omega(y,t) \le \frac{x_2 + \widehat r}{r_0-\eps} \frac{M}{\eps^2|\log\eps|} \quad \forall\,y \colon |y_2-x_2| \le \widehat r\,.
\]
Therefore, by repeating the same reasoning done for the analysis of the term $T_{1,2}$ in Lemma \ref{lem:2}, in this case we obtain
\[
\int_{|y_2 - x_2| \le \widehat r} \! \rmd y\, \frac{\omega_\eps(y,t)}{2\pi|x-y|} \le \frac{C\sqrt{1+x_2}}{\eps |\log\eps|}\,.
\]
On the other hand, using again Eq.~\eqref{cons-mass},
\[
\int_{|y_2 - x_2| > \widehat r} \! \rmd y\, \frac{\omega_\eps(y,t)}{2\pi|x-y|} \le \frac1{2\pi \widehat r|\log\eps|}\,. 
\]
Therefore,
\begin{equation}
\label{utilde1}
|\widetilde u_1 (x,t)|  \le \frac{C\sqrt{1+x_2}}{\eps |\log\eps|}\,,
\end{equation}
whence, as $1+x_2 \le 2+x_2^2$, from the Cauchy-Schwarz inequality and Eqs.~\eqref{M2}, \eqref{M2<*}, and \eqref{cons-mass},
\[
\left||\log\eps| \int_{|x_2 - r_0| > \widehat r} \!\rmd x\, \omega_\eps(x,t) \, \widetilde u_1(x,t)\right|  \le \frac{C}{\eps\sqrt{|\log\eps|}} \sqrt{\int_{|x_2 - r_0| > \widehat r} \!\rmd x\, \omega_\eps(x,t)}\,,
\]
which vanishes as $\eps \to 0$ by Eq.~\eqref{smt*}. Next, from Eqs.~\eqref{sing_Fc}, \eqref{reg_F}, \eqref{picc_F}, \eqref{cons-mass}, and \eqref{sR},
\begin{align}
\label{RF}
& \left| |\log\eps|\int_{|x_2 - r_0| \le \widehat r} \!\rmd x\, \omega_\eps(x,t) \left( \int\!\rmd y\,  \mc R_1(x,y) \omega_\eps(y,t) + F^\eps_1 (x,t)\right)\right| \nonumber \\ & \qquad \le C \left(\int\!\rmd y\, \sqrt{y_2} (1+|\log y_2|) \omega_\eps(y,t) + \frac{C}{|\log\eps|} \right) \le \frac{C}{|\log\eps|}\,,
\end{align}
where in the last inequality we also used that $\sqrt{y_2} (1+|\log y_2|) \le C(1+y_2^2)$ and again Eqs.~\eqref{M2}, \eqref{M2<*}, and \eqref{cons-mass}.

From the above estimates and Eq.~\eqref{bc_bound} we have thus shown that
\begin{equation}
\label{b12}
\lim_{\eps \to 0} |\dot B_{\eps,1}(t) - Q_\eps(t)|  = 0 \quad \forall\, t\in [0, T] \,,
\end{equation}
where
\[
Q_\eps(t) := |\log\eps|\int_{|x_2 - r_0| \le \widehat r}\!\rmd x\, \omega_\eps(x,t) \frac{1}{4\pi x_2} \int\!\rmd y\, \log\frac {1+|x-y|}{|x-y|} \omega_\eps(y,t) \,.
\]
To compute the limit of $Q_\eps(t)$ as $\eps\to 0$, we first write
\[
Q_\eps(t) = Q_\eps^1(t) + Q_\eps^2(t)\,,
\]
where
\[
\begin{split}
Q_\eps^1(t) & =  |\log\eps|\int_{|x_2 - r_0| \le \widehat r}\!\rmd x\, \omega_\eps(x,t) \frac{1}{4\pi x_2} \int_{|y_2 - x_2| \le  2 \widehat r}\!\rmd y\, \log\frac {1+|x-y|}{|x-y|} \omega_\eps(y,t)\,,  \\ Q_\eps^2(t) & =  |\log\eps|\int_{|x_2 - r_0| \le \widehat r}\!\rmd x\, \omega_\eps(x,t) \frac{1}{4\pi x_2} \int_{|y_2 - x_2| > 2 \widehat r}\!\rmd y\, \log\frac {1+|x-y|}{|x-y|} \omega_\eps(y,t)\,.
\end{split}
\]
We observe that the function $\rho \mapsto \log\frac {1+\rho|}{\rho}$ is decreasing and  diverging as $\rho\to 0^+$. In particular,
\begin{align}
\label{Q2}
Q_\eps^2(t) & \le  \frac{ |\log\eps| }{4\pi(r_0-\widehat r)} \log\frac{1+ 2\widehat r}{2\widehat r} \int_{|x_2 - r_0| \le \widehat r}\!\rmd x\, \omega_\eps(x,t) \int_{|y_2 - x_2| > 2 \widehat r}\!\rmd y\, \omega_\eps(y,t) \nonumber \\ & \le  \frac{1}{2\pi(r_0-\widehat r)} \log\frac{1+ 2\widehat r}{2\widehat r} \int_{|y_2 - r_0| > \widehat r} \!\rmd y\, \omega_\eps(y,t)\,,
\end{align}
where in the last inequality we used that 
\[
\{(x,y) \colon |x_2 - r_0| \le \widehat r \,, |y_2 - x_2| > 2 \widehat r \} \subset \{(x,y) \colon |y_2 - r_0| > \widehat r\}\,.
\]
Therefore, by Eq.~\eqref{smt*}, $Q_\eps^2(t) \to 0$ as $\eps\to 0$, so that it remains to compute the limit of $Q_\eps^1(t)$ as $\eps\to 0$. To this purpose, again using the abbreviated notation  $\Sigma_t = \Sigma(q_\eps(t), \varrho_\eps)$, we decompose 
\[
Q_\eps^1(t) = Q_\eps^{1,1}(t) + Q_\eps^{1,2}(t)\,,
\]
where
\[
Q_\eps^{1,1}(t) = |\log\eps| \int_{\Sigma_t}\!\rmd x\, \omega_\eps(x,t)  \frac{1}{4\pi x_2} \int_{\Sigma_t}\!\rmd y\, \log\frac {1+|x-y|}{|x-y|} \omega_\eps(y,t)\,.
\]
By the first limit in Eq.~\eqref{qq}, the disk $\Sigma(q_\eps(t), \varrho_\eps)$ is internal to the strip $\{x\colon |x_2-r_0|\le \widehat r\}$ for any $\eps$ small enough, therefore the rest $Q_\eps^{1,2}(t) = Q_\eps^1(t) - Q_\eps^{1,1}(t)$ is the sum of three terms, each one is the integration of the same function\footnote{$\id_A$ denotes the indicator function of the set $A$.} 
\[
\mc F(x,y) := \frac{1}{4\pi x_2} \log\frac {1+|x-y|}{|x-y|} \omega_\eps(x,t)\omega_\eps(y,t) \id_{\{(x,y) \colon |x_2 - r_0| \le \widehat r\,, \;\; |y_2 - x_2| \le  2 \widehat r\}} \,,
\]
and in each domain of integration at least one between the $x$ and the $y$ variable is contained in $\Sigma(q_\eps(t),  \varrho_\eps)^\complement$. Moreover, using also Eq.~\eqref{rh<r_0/4},
\[
\frac{1}{y_2} \le \frac{5}{x_2} \le \frac{3r_0+4}{3r_0}\frac{5}{x_2+1} \quad \forall\, (x,y) \colon |x_2 - r_0| \le \widehat r\,, \;\; |y_2 - x_2| \le  2 \widehat r\,.
\]
Therefore,
\[
Q_\eps^{1,2}(t) \le C|\log\eps| \int_{\Sigma_t^\complement}\!\rmd x\,\frac{\omega_\eps(x,t)}{x_2+1}  \int_{|y_2 - x_2| \le  2 \widehat r}\!\rmd y\, \log\frac {1+|x-y|}{|x-y|} \omega_\eps(y,t)\,.
\]
The maximum of the $\rmd y$-integral in the right-hand side is achieved when we rearrange the mass as close as possible to the singularity. As by Eq.~\eqref{maxpri}, \eqref{MgammaF}, and \eqref{initialF},
\[
\omega_\eps(y,t) \le \frac{C(x_2+1)}{\eps^2|\log\eps|}  \quad \forall\, y \colon |y_2 - x_2| \le  2 \widehat r\,,
\]
if $\bar\rho$ is such that $C(x_2+1) \bar\rho^2 /(\eps^2|\log\eps|)= 1/|\log\eps|$ then
\begin{align}
\label{Q12}
& \int_{|y_2 - x_2| \le  2 \widehat r}\!\rmd y\, \log\frac {1+|x-y|}{|x-y|}\, \omega_\eps(y,t) \le \frac{C(x_2+1)}{\eps^2|\log\eps|} \int_0^{\bar\rho}\!\rmd \rho\, \rho \, \log\frac{1+\rho}{\rho} \nonumber \\ & \qquad =  \frac{C(x_2+1)}{\eps^2|\log\eps|} \bigg\{\frac{\bar\rho^2}{2} \log\frac{1+\bar\rho}{\bar\rho} + \frac 12 \int_0^{\bar\rho}\!\rmd \rho\, \frac{\rho}{1+\rho} \bigg\} \le C\log(x_2+1)\,.
\end{align}
Therefore,
\[
Q_\eps^{1,2}(t) \le C |\log\eps| \int_{\Sigma_t^\complement}\!\rmd x\, \omega_\eps(x,t)\,,
\]
which implies $Q_\eps^{1,2}(t) \to 0 $ as $\eps\to 0$ by Eq.~\eqref{lem1b}.

Concerning $Q_\eps^{1,1}(t)$, we obtain a lower bound by inserting a lower bound to the function $\frac{1}{4\pi x_2}\log\frac {1+|x-y|}{|x-y|}$ in the domain of integration and applying again Eq.~\eqref{lem1b}, 
\begin{align}
\label{q2}
Q_\eps^{1,1}(t) & \ge \frac{|\log\eps|}{4\pi (q_{\eps,2}(t)+\varrho_\eps)} \log\frac {1+2\varrho_\eps}{2\varrho_\eps} \bigg(\int_{\Sigma_t}\!\rmd x\, \omega_\eps(x,t)\bigg)^2 \nonumber \\ & \ge  \frac{|\log\eps|}{4\pi (q_{\eps,2}(t)+\varrho_\eps)} \log\frac {1+2\varrho_\eps}{2\varrho_\eps} \frac{1}{|\log \eps|^2}\bigg(1-\frac{C_1}{|\log\eps|^{\eta-\gamma}} \bigg)^2 \nonumber \\  & = \frac{1}{4\pi (q_{\eps,2}(t)+\varrho_\eps)} \big[1 + o_-(\eps)\big]\,,
\end{align}
where in the last equality we used that $\varrho_\eps = \varrho_1\eps \exp(|\log\eps|^\eta)$ and $o_-(\eps)\to 0$ as $\eps\to 0$. On the other hand, again by rearrangement and noticing that Eq.~\eqref{maxpri}, \eqref{MgammaF}, and \eqref{initialF} now imply there is $M_1>0$ such that, for any $\eps$ small enough,
\[
\omega_\eps(y,t) \le \frac{M_1}{\eps^2|\log\eps|} \quad \forall\, y\in \Sigma_t\,,
\]
if $\bar\rho$ is such that $\pi\bar\rho^2 M_1/(\eps^2|\log\eps|) = 1/|\log\eps|$ we have
\begin{align}
\label{q3}
Q_\eps^{1,1}(t) & \le \frac{1}{4\pi (q_{\eps,2}(t)-\varrho_\eps)} \sup_x \int_{\Sigma_t}\!\rmd y\, \log\frac {1+|x-y|}{|x-y|} \omega_\eps(y,t) \nonumber \\ & \le  \frac{M_1}{2\eps^2 |\log\eps|(q_{\eps,2}(t)-\varrho_\eps)}  \bigg\{\frac{\bar\rho^2}{2} \log\frac{1+\bar\rho}{\bar\rho} + \frac 12 \int_0^{\bar\rho}\!\rmd \rho\, \frac{\rho}{1+\rho} \bigg\} \nonumber \\ & = \frac{1}{4\pi (q_{\eps,2}(t)-\varrho_\eps)} \big[1 + o_+(\eps)\big]\,, 
\end{align}
with $o_+(\eps)\to 0$ as $\eps\to 0$. By the first limit in Eq.~\eqref{qq}, we conclude that the right-hand side in both Eqs.~\eqref{q2} and \eqref{q3} converges to $1/(4\pi r_0)$ as $\eps\to 0$, so that Eq.~\eqref{b1c} follows from Eq.~\eqref{b12} and \eqref{initialF}.

\smallskip
We are left with the proof of Eq.~\eqref{Jst}. We compute the time derivative of $J_\eps(t)$, by using Eq.~\eqref{weqF},
\[
\begin{split}
\dot J_\eps(t) & = 2\int \!\rmd x\, \omega_\eps(x,t) (u_1(x,t)+F_1^\eps(x,t) - \dot B_{\eps,1}(t)) (x_1-B_{\eps,1}(t)) \\ & \quad + 2 \nu  \int \!\rmd x\, \omega_\eps(x,t) \\ & =  2 \int \!\rmd x\, \omega_\eps(x,t) (u_1+F_1^\eps)(x,t)(x_1-B_{\eps,1}(t)) + 2 \nu  \int \!\rmd x\, \omega_\eps(x,t)  \\ & = 2 \int \!\rmd x\, \omega_\eps(x,t) (u_1+F_1^\eps - \widetilde u_1)(x,t) (x_1-B_{\eps,1}(t)) \\ & \quad +  2 \int \!\rmd x\, \omega_\eps(x,t) \widetilde u_2(x,t) x_2 + 2 \nu  \int \!\rmd x\, \omega_\eps(x,t)\,,
\end{split}
\]
where in the second equality we used that $\int \rmd x\, \omega_\eps(x,t)\, (x_1-B_{\eps,1}(t)) = 0$ (by definition of $B_{\eps,1}(t)$), and the third equality follows from Eq.~\eqref{tu=} and the identity
\[
\int\! \rmd x \, \omega_\eps(x,t) \, x \cdot \widetilde{u}(x,t) = 0\,.
\]
The latter comes from the fact that, as $(x-y) \cdot K(x-y) =0$, 
\[
\begin{split} 
\int\! \rmd x \, \omega_\eps(x,t) \, x \cdot \widetilde{u}(x,t) = & \int\! \rmd x \int\! \rmd y \, \omega_\eps(x,t)\,\omega_\eps(y, t) \, x \cdot K(x-y) \\ = & \int\! \rmd x \int\! \rmd y \, \omega_\eps(x,t)\,\omega_\eps(y, t) \, y \cdot K(x-y)\,,
\end{split}
\]
whence this integral is zero by the anti-symmetry of $K$.

Recalling Eq.~\eqref{decom_u} and using Eq.~\eqref{M2d=0} the time derivative of $J_\eps(t)$ reads,
\begin{equation}
\label{jdot}
\dot J_\eps(t) = \mc J_\eps^1(t) + \mc J_\eps^2(t) + \mc J_\eps^3(t) + 2 \nu  \int \!\rmd x\, \omega_\eps(x,t)\,,
\end{equation}
where
\[
\begin{split}
\mc J_\eps^1(t) & = 2 \int_{|x_2-r_0| > \widehat r} \!\rmd x\, \omega_\eps(x,t) (u_1+F_1^\eps - \widetilde u_1)(x,t) (x_1-B_{\eps,1}(t))\,, \\ \mc J_\eps^2(t) & = 2 \int_{|x_2-r_0| \le \widehat r} \!\rmd x\, \omega_\eps(x,t)(u_1+F_1^\eps - \widetilde u_1)(x,t)  (x_1-B_{\eps,1}(t))\,,  \\ \mc J_\eps^3(t) & = -  2 \int \!\rmd x\, \omega_\eps(x,t) x_2 \int\! \rmd y\, \omega_\eps(y,t) \, \mc R_2(x,y)\,.
\end{split}
\]

By Eq.~\eqref{cons-mass} and recalling $\nu\le \eps^2|\log\eps|^\gamma$, the last term in the right-hand side of Eq.~\eqref{jdot} is very small,
\[
2 \nu  \int \!\rmd x\, \omega_\eps(x,t)  \le 2 \eps^2|\log\eps|^{\gamma-1}\,.
\]
Next, as already observed just after Eq.~\eqref{bpunto}, from Eqs.~\eqref{sing_F}, \eqref{reg_F}, \eqref{picc_F}, and \eqref{uinf<}, we have $|u_1|+|F_1^\eps| \le C/(\eps|\log\eps|)$, while $|\widetilde u_1|$ is bounded in Eq.~\eqref{utilde1}. Therefore, by applying the Cauchy-Schwarz inequality twice and then Eqs.~\eqref{M2}, \eqref{M2<*}, and \eqref{cons-mass},
\[
\begin{split}
|\mc J_\eps^1(t) | & \le  \frac{C\sqrt{J_\eps(t)}}{\eps |\log\eps|} \sqrt{\int_{|x_2 - r_0| > \widehat r} \!\rmd x\, \omega_\eps(x,t) (1+x_2)} \\ & \le \frac{C\sqrt{J_\eps(t)}}{\eps |\log\eps|^{5/4}} \left[\int_{|x_2 - r_0| > \widehat r} \!\rmd x\, \omega_\eps(x,t) \right]^{1/4} \le  C_\ell \eps^\ell \sqrt{J_\eps(t)} \;\; \forall\, t\in [0,T] \;\;\forall\, \ell>0\,,
\end{split}
\]
where Eq.~\eqref{smt*} has been applied in the last inequality. On the other hand, again by the Cauchy-Schwarz inequality,
\[
|\mc J_\eps^2(t)| \le  \sqrt{J_\eps(t)} \, \sqrt{\int_{|x_2-r_0| \le \widehat r} \!\rmd x\, \omega_\eps(x,t) (u_1+F_1^\eps - \widetilde u_1)^2(x,t)}\,,
\]
where, by Eq.~\eqref{decom_u},
\[
\begin{split}
& |(u_1+F_1^\eps - \widetilde u_1)(x,t)| \le \left|\int\!\rmd y\,  \mc R_1(x,y) \omega_\eps(y,t) + F^\eps_1 (x,t)\right| \\ & \quad + \int_{|y_2-x_2|> 2\widehat r}\!\rmd y\,  L_1(x,y) \omega_\eps(y,t) + \int_{|y_2-x_2|\le 2\widehat r}\!\rmd y\,  L_1(x,y) \omega_\eps(y,t)\,.
\end{split}
\]
In Eq.~\eqref{RF} we have shown that for $|x_2-r_0| \le \widehat r$ the first term in the right-hand side is bounded by $C/|\log\eps|$. Similarly, in Eq.~\eqref{Q2} we obtained an upper bound for the second term when $|x_2-r_0| \le \widehat r$. More precisely,
\[
\begin{split}
\int_{|y_2-x_2|> 2\widehat r}\!\rmd y\,  L_1(x,y) \omega_\eps(y,t) & \le \frac{1}{2\pi(r_0-\widehat r)} \log\frac{1+ 2\widehat r}{2\widehat r} \int_{|y_2 - r_0| > \widehat r} \!\rmd y\, \omega_\eps(y,t) \\ & \le  C_\ell \eps^\ell \quad \forall\, t\in [0,T] \quad\forall\, \ell>0\,,
\end{split}
\]
where Eq.~\eqref{smt*} has been applied in the last inequality. Finally, from Eq.~\eqref{Q12} we deduce that if $|x_2-r_0| \le \widehat r$ then
\[
\int_{|y_2-x_2|\le 2\widehat r}\!\rmd y\,  L_1(x,y) \omega_\eps(y,t) \le C\,.
\]
From the above estimates and Eq.~\eqref{cons-mass} we conclude that
\[
|\mc J_\eps^2(t)| \le \frac{C}{|\log\eps|^{1/2}} \sqrt{J_\eps(t)} \quad \forall\, t\in [0,T]\,.
\]
Finally, by Eqs.~\eqref{cons-mass} and \eqref{sR},
\[
|\mc J_\eps^3(t)| \le \frac{C}{|\log\eps|^2}\,.
\]

Collecting all together we thus proved that
\[
|\dot J_\eps(t)| \le \frac{C}{|\log\eps|^{1/2}} \sqrt{J_\eps(t)} + \frac{C}{|\log\eps|^2} \quad \forall\, t\in [0,T]\,.
\]
Since the initial data imply $J_\eps(0) \le 4\eps^2$, we deduce Eq.~\eqref{Jst} from the above inequality by arguing as done in the proof of Eq.~\eqref{Iee}. More precisely, if
\[
T_M = \sup\left\{ t\in [0,T]\colon J_\eps(s) \le \frac{M}{|\log\eps|} \;\;  \forall\, s\in [0,t] \right\}
\]
then $T_M>0$ for any $\eps$ small enough and
\[
|\dot J_\eps(t)| \le \frac{C\sqrt M}{|\log\eps|} + \frac{C}{|\log\eps|^2} \quad \forall\, t\in [0,T_M]\,.
\]
This implicates, as $T_M \le T$,
\[
J_\eps(t) \le 4\eps^2 + \frac{CT\sqrt M}{|\log\eps|} + \frac{CT}{|\log\eps|^2} \quad \forall\, t\in [0,T_M]\,.
\]
If $M$ is chosen large enough then, for any $\eps$ small enough,
\[
J_\eps(t) \le \frac{M}{2|\log\eps|}  \quad \forall\, t\in [0,T_M]\,,
\]
which implies $T_M=T$ by continuity, whence Eq.~\eqref{Jst} (with $C=M$).
\end{proof}

\section{Proof of Theorem \ref{thm:1}}
\label{sec:5}

As already explained at the end of Section \ref{sec:2}, the motion of each vortex $\omega_{i,\eps}(x,t)$ can be viewed as the evolution of a single vortex ring driven by the sum of the velocity generated by $\omega_{i,\eps}(x,t)$ itself plus the external time-depending field given by Eq.~\eqref{Fconj}.

Recalling Eqs.~\eqref{initial} and \eqref{2D} we fix $\chi_0 \gg 1$ and define
\begin{equation}
\label{ansatz_massa_vort_i}
T^0_\eps := \bigcap_{k=1}^N\sup\left\{t\in [0,T] \colon \int_{|x_2-r_k|> D/2} \!\rmd x \, \omega_{k,\eps}(x,t) \le \eps^{\chi_0} \;\; \forall\, s\in [0,t] \right\},
\end{equation}
so that $T^0_\eps>0$ for any $\eps$ small enough by continuity. For each $i=1,\dots, N$, we decompose the field Eq.~\eqref{Fconj} as
\[
F^{\eps,i}(x,t) = G^{\eps,i}(x,t) + \bar F^{\eps,i}(x,t)\,,
\]
with
\begin{align*}
G^{\eps,i}(x,t) & = \ \sum_{j\ne i} \int_{|y_2-r_j|\le D/2}\!\rmd y\, H(x,y)\, \omega_{j,\eps}(y,t)\,, \\ \bar F^{\eps,i}(x,t) & = \sum_{j\ne i} \int_{|y_2-r_j| > D/2}\!\rmd y\, H(x,y)\, \omega_{j,\eps}(y,t)\,.
\end{align*}
Since $G^{\eps,i}(x,t)$ is generated by vorticities supported in disjoint strips, we can apply the same reasoning of Ref.~\cite[Lemma 2.3]{ButCavMar} to conclude that there exists a field $\widetilde F^{\eps,i}(x,t)$ such that 

\noindent
(a) $\partial_{x_1}(x_2 \widetilde F^{i,\eps}_1) + \partial_{x_2}(x_2 \widetilde F^{i,\eps}_2) = 0$;

\noindent
(b) for a suitable $\widetilde  C>0$, Eq.~\eqref{reg_F} holds with $\widetilde F^\eps(x,t) = \widetilde F^{i,\eps}(x,t)$;

\noindent
(c) $\widetilde F^{i,\eps}(x,t) = G^{i,\eps}(x,t)$ for any $(x,t)$ such that $|x_2-r_i|\le D/2$ and $t\in [0,T]$.

\smallskip
We then decompose $F^{\eps,i}(x,t)$ as in Eq.~\eqref{ext_field} with
\[
\widehat F^\eps(x,t) = G^{\eps,i}(x,t) - \widetilde F^{\eps,i}(x,t)\,, \quad \widetilde F^\eps(x,t) = \widetilde F^{\eps,i}(x,t) \quad \bar F^\eps(x,t) = \bar F^{\eps,i}(x,t)\,. 
\]
We claim that for any $t\in [0,T^0_\eps]$ the above functions satisfy the assumptions Eqs.~\eqref{divfree}, \eqref{sing_Fc}, \eqref{sing_F}, \eqref{rh<r_0/4}, \eqref{reg_F}, and \eqref{picc_F}, with $r_0=r_i$. Indeed, by item (a) above, the condition Eq.~\eqref{divfree} is clearly satisfied. Moreover:

(i) Eq.~\eqref{reg_F} holds for any $t\in [0,T]$ by the above item (b).

(ii) The support property Eq.~\eqref{sing_Fc} clearly holds with $\widehat r = D/2$, which also satisfies Eq.~\eqref{rh<r_0/4} thanks to Eq.~\eqref{2D}. Moreover, in view of Eqs.~\eqref{cons-mass}, \eqref{maxpri} and the conservation of  $M_2 =\|x_2^2 \omega_\eps\|_{L^1}$, Eq.~\eqref{u_piccolo} and the initial conditions Eqs.~\eqref{MgammaF} and \eqref{initialF} imply that $G^{\eps,i}(x,t)$ (and therefore also $\widehat F^\eps(x,t) $) satisfies Eq.~\eqref{sing_F} for any $t\in [0,T]$ and a suitable constant $\widehat C>0$.

(iii) Finally, from the same reasoning done to estimate $G^{\eps,i}(x,t)$ and using the definition Eq.~\eqref{ansatz_massa_vort_i} we deduce that $|\bar F^{\eps,i}(x,t)| \le C \eps^{\chi_0/4}/(\eps|\log\eps|^{3/4})$ for any $t\in [0,T^0_\eps]$, i.e., Eq.~\eqref{picc_F} holds with $\beta>1$ (provided that, e.g., $\chi_0>8$).

At this point, the analysis of the previous sections can be applied to each vortex $i=1,\ldots, N$, choosing $T_\eps = T^0_\eps$ in Eq.~\eqref{ansatz_massa_vort}. In particular, Corollary \ref{cor:1} implies $T^0_\eps = T$ for any $\eps$ small enough, and Theorem \ref{thm:1} follows from Theorem \ref{thm:2}.
\qed

\appendix

\section{Proof of Lemma \ref{lem:3}}
\label{sec:A}

Given $R,h$ such that $\widehat r/2 \ge R+h \ge R \ge 2h^\alpha$ ($R,h$ will be eventually chosen vanishing as $\eps\to 0$), with $\alpha$ a positive parameter to be fixed later, let $W_{R,h}(x_2)$, with $x_2$ the second component of $x = (x_1, x_2)$, be a non-negative smooth function, such that
\begin{equation}
\label{W1}
W_{R,h}(x_2) = \begin{cases} 1 & \text{if $|x_2|\le R$}, \\ 0 & \text{if $|x_2|\ge R+h$}, \end{cases}
\end{equation}
and its first and second derivative satisfy
\begin{equation}
\label{W2}
| W_{R,h}'(x_2)| < \frac{C}{h}\,,
\end{equation}
\begin{equation}
\label{W3}
| W_{R,h}''(x_2)| < \frac{C}{h^2}\,.
\end{equation}
The quantity
\begin{equation}
\label{mass 1}
\mu_t(R,h) = \int\! \rmd x \, \big[1-W_{R,h}(x_2-B_{\eps,2}^*(t))\big]\, \omega_\eps (x,t)
\end{equation}
is a mollified version of $m_t$ satisfying
\begin{equation}
\label{2mass 3}
\mu_t(R,h) \le m_t(R) \le \mu_t(R-h,h)\,,
\end{equation}
hence it is sufficient to prove Eq.~\eqref{smt} with $\mu_t$ in place  of $m_t$.

We observe that from Eq.~\eqref{Br0} it follows that if $\eps$ is small enough then 
\begin{equation}
\label{suppW}
W_{R,h}(x_2-B_{\eps,2}^*(t)) = 0 \quad \forall\, x_2 \notin [r_0-\widehat r, r_0+\widehat r]  \quad \forall\, t\in [0,T_\eps]\,,
\end{equation}
and throughout the rest of the proof we shall assume $\eps$ such that Eq.~\eqref{suppW} holds. In particular, we can apply Eq.~\eqref{weqF} with test function $f(x,t) = W_{R,h}(x_2-B_{\eps,2}^*(t))$ to compute the time derivative of $\mu_t(R,h)$. Using also Eqs.~\eqref{dotom<}, \eqref{decom_u}, and \eqref{growth B}, we obtain,
\begin{align}
\label{mu_t}
\frac{\rmd}{\rmd t} \mu_t(R,h) & \le - \int\! \rmd x\, \omega_\eps(x,t)  \left(u_2(x,t)+F^\eps_2(x,t)-\dot{B}_{\eps,2}^*(t)\right) W_{R,h}'(x_2-B_{\eps,2}^*(t))  \nonumber \\ & \quad - \nu \int\! \rmd x\,  \omega_\eps(x,t) \left[ W_{R,h}''(x_2-B_{\eps,2}^*(t)) - \frac{1}{x_2}  W_{R,h}'(x_2-B_{\eps,2}^*(t))  \right] \nonumber \\ & \le  - \mc H_1 -  \mc H_2 - \mc H_3\,,
\end{align}
with
\begin{equation*}
\begin{split}
\mc H_1 & = \int\! \rmd x\,   W_{R,h}'(x_2-B_{\eps,2}^*(t))  \int\!\rmd y \, K_2(x-y)\, \omega_\eps(y,t)\, \omega_\eps(x,t) \\ & = \frac 12 \int\! \rmd x \! \int\! \rmd y\, \omega_\eps(x,t)\,  \omega_\eps(y,t) \\ & \quad \times \left[ W_{R,h}'(x_2-B_{\eps,2}^*(t)) -   W_{R,h}'(y_2-B_{\eps,2}^*(t))\right]   K_2(x-y) \,, \\ \mc H_2 & = \int\! \rmd x\, W_{R,h}'(x_2-B_{\eps,2}^*(t))\, \omega_\eps(x,t) V(x,t)\,, \\ \mc
H_3 & =   \nu \int\! \rmd x \left[ W_{R,h}''(x_2-B_{\eps,2}^*(t))  -
\frac{1}{x_2}  W_{R,h}'(x_2-B_{\eps,2}^*(t))  \right] \omega_\eps(x,t)
\end{split}
\end{equation*}
where we used the antisymmetry of $K$ to get the second expression of $\mc H_1$, and 
\begin{equation}
\label{j}
V(x,t) = F^\eps_2(x,t)  +\int\!\rmd z\, \mc R_2(x,z) \omega_\eps(z,t) - \dot B_{\eps,2}^*(t)\,.
\end{equation}

The term $\mc H_1$ is equal to the term $H_3$ appearing in Ref.~\cite[Eq.~(3.39)]{ButCavMar}, except for $B_{\eps,2}^*(t)$ in place of $B_{\eps,2}(t)$ (defined in Eq.~\eqref{c.m.}). It can be estimated exactly as done for $H_3$ in Ref.~\cite{ButCavMar}, getting
\[
|\mc H_1| \le \frac{C}{h^{1+\alpha} |\log\eps|} m_t(R) + \frac{C \widetilde I_\eps(t)}{ h^2 R^2}m_t(R) \quad \forall\, t\in [0,T_\eps]\,, 
\]
where
\[
\widetilde I_\eps(t) = \int\! \rmd x\,  \omega_\eps(x,t) \left(x_2-B_{\eps,2}^*(t)\right)^2.
\]
From definitions Eqs.~\eqref{moment*} and \eqref{G},
\[
|\widetilde I_\eps(t) - I_\eps^*(t)| \le C \eps^\chi \qquad \forall\, t\in [0,T_\eps]\,,
\]
so that, in view of Eq.~\eqref{Iee},
\begin{equation}
\label{a1s}
|\mc H_1| \le  C \left( \frac{1}{h^{1+\alpha}|\log\eps|} + \frac{1}{h^2 R^2 |\log\eps|^2} \right) m_t(R) \quad \forall\, t\in [0,T_\eps]\,.
\end{equation}

Concerning $\mc H_2$, we note that Eqs.~\eqref{sing_Fc}, \eqref{reg_F}, \eqref{picc_F}, \eqref{cons-mass}, \eqref{sR}, and \eqref{boundB} imply
\begin{equation}
\label{distance4}
|V(x,t)| \leq \frac{C}{|\log\eps|} \quad \forall\, x_2 \in [r_0-\widehat r, r_0+\widehat r]  \quad \forall\, t\in [0,T_\eps]\,.
\end{equation}
On the other hand, in view of Eqs.~\eqref{W1} and \eqref{suppW},  $W_{R,h}'(x_2-B_{\eps,2}^*(t))$ is zero if $|x_2-B_{\eps,2}(t)| \ge R$ as well as if $x_2 \in [r_0-\widehat r, r_0+\widehat r]$ and $t\in [0,T_\eps]$. Therefore, from Eqs.~\eqref{W2} and \eqref{distance4},
\begin{equation}
\label{acca4}
|\mc H_2| \le \frac{C}{h |\log\eps|} m_t(R) \quad \forall\, t\in [0,T_\eps]\,.
\end{equation}

Similarly, $W_{R,h}''(x_2-B_{\eps,2}^*(t))$ is zero if $|x_2-B_{\eps,2}(t)| \ge R$ as well as if $x_2 \in [r_0-\widehat r, r_0+\widehat r]$ and $t\in [0,T_\eps]$, so that, from Eq.~\eqref{W3},
\begin{equation}
\label{H_5}
|\mc H_3| \le C  \nu \left( \frac{1}{h}+\frac{1}{h^2} \right) m_t(R) \quad \forall\, t\in [0,T_\eps]\,.  
\end{equation}

\noindent From Eqs.~\eqref{a1s}, \eqref{acca4} and \eqref{H_5}, recalling Eq.~\eqref{mu_t} and that $\nu \leq \eps^2 |\log\eps|^\gamma$,  we conclude that
\begin{equation}
\label{equ_mm}
\frac{\rmd}{\rmd t} \mu_t (R,h) \leq A_\eps(R, h) m_t(R) \quad \forall\, t\in [0,T_\eps]\,,
\end{equation}
where
\[
A_\eps(R, h) = C \left(\frac{1}{h^{1+\alpha} |\log\eps|} + \frac{1}{h^2 R^2 |\log\eps|^2} +\frac{1}{h |\log\eps|}+ \frac{\eps^2 |\log\eps|^\gamma}{h^2}\right).
\]
Therefore, by Eqs.~\eqref{2mass 3} and \eqref{equ_mm},
\begin{equation}
\label{iter}
\mu_t (R,h) \le \mu_0 (R,h) + A_\eps(R, h) \int_0^t\! \rmd s\, \mu_s (R-h, h)\,.
\end{equation}

The proof can now be concluded arguing as in Ref.~\cite{ButCavMar}. Nonetheless, to keep the proof sufficiently self-contained, we prefer to report the details.

We assume $\eps$ sufficiently small and iterate Eq.~\eqref{iter} $n=\lfloor|\log\eps|\rfloor$  times (denoting with $\lfloor a\rfloor$ the integer part of $a>0$), from
\begin{equation}
\label{range_R}
R_0 =\frac{1}{|\log\eps|^{k}}  \quad\text{to} \quad R_n = \frac{1}{2|\log\eps|^{k}}\,,
\end{equation}
where
\[
R_n = R_0 -n h\,, \quad h = \frac{1}{2 n |\log\eps|^{k}} \,.
\]
Recalling the condition $R\ge 2 h^\alpha$, stated at the beginning of the proof and under which Eq.~\eqref{equ_mm} has been deduced, we make the choice
\begin{equation}
\label{alpha_delta}
\alpha=\frac{1-k}{1+k}-\delta \, , \qquad \delta\in \left(0, \frac{1-2k}{1+k}\right)\,.
\end{equation}
This implies
\[
h^\alpha \approx C \left(  \frac{1}{|\log\eps|^{1+k}} \right)^\alpha = C \left(  \frac{1}{|\log\eps|^{1+k}} \right)^{\frac{1-k}{1+k}-\delta} = \frac{C}{|\log\eps|^{1-k-(1+k)\delta}}\,,
\]
with $k<1-k-(1+k)\delta$ (by the choice of $\delta$ in Eq.~\eqref{alpha_delta}), hence $h^\alpha \ll R$ if $\eps$ is small enough and $R$ is in the range established by Eq.~\eqref{range_R}. Moreover, for $\eps$ small,
\[
\begin{split}
\frac{1}{h^{1+\alpha} |\log\eps|} & \le C \frac{\left( |\log\eps|^{k+1}   \right)^{1+\alpha}}{|\log\eps|} \le C |\log\eps |^{1-\delta(k+1)} \,, \\ 
\frac{\eps^2|\log\eps|^\gamma}{h^2} \leq \frac{1}{h^2 R^2 |\log\eps|^2} & \le C \frac{|\log\eps|^{4k+2} }{|\log\eps|^2} \leq |\log\eps|^{4k}\,, \\ \frac{1}{h |\log\eps|} & \le C |\log\eps|^{k} \,,
\end{split}
\]
so that there is $q \in (0,1)$ such that $A_\eps(R,h) \le C |\log\eps |^q$. In conclusion, for any $\eps$ small enough,
\[
\begin{split}
\mu_t(R_0-h,h) & \le \mu_0(R_0-h,h) + \sum_{j=1}^{n-1} \mu_0(R_j,h) \frac{(C |\log\eps|^q t)^j}{j!} \\ & \quad + \frac{(C |\log\eps|^q )^{n}}{(n-1)!} \int_0^t\!{\textnormal{d}} s\,  (t-s)^{n-1}\mu_s(R_{n},h) \,.
\end{split}
\]
Since $\Lambda_\eps(0) \subset \Sigma(z|\eps)$, if $\eps$ is sufficiently small then $\mu_0(R_j,h)=0$ for any $j=0,\ldots,n$, hence, 
\begin{equation}
\label{mass 15'}
\mu_t(R_0-h,h) \le \frac{(C |\log\eps|^q)^{n}}{(n-1)!} \int_0^t\! \rmd s\,  (t-s)^{n-1}\mu_s(R_{n},h) \le  \frac{(C |\log\eps|^q  t)^{n}}{n!}\,,
\end{equation}
where in the last inequality we have used the trivial bound $\mu_s(R_{n},h) \le 1$. Therefore, using Eq.~\eqref{2mass 3},  Stirling formula, and $n=\lfloor|\log\eps|\rfloor$,
\[
m_t(R_0) \le \mu_t(R_0 -h,h) \le  \frac{C}{|\log\eps|^{(1-q)|\log\eps|}} \,,
\]
which implies Eq.~\eqref{smt}.
\qed

\section{Proof of Lemma \ref{lem:4}}
\label{sec:B}

In absence of external field, the proof of Eq.~\eqref{lem1b} given in Ref.~\cite{BruMar} is obtained by applying the abstract concentration result Ref.~\cite[Lemma 2.1]{BruMar} to the inequality Ref.~\cite[Eq.~(2.9)]{BruMar}, involving an integral functional of the vorticity.\footnote{A notation warning: in Ref.~\cite{BruMar}, $\eps \to \sigma$ and $\nu \le \eps^2|\log\eps|^\gamma \to \nu \le \sigma^2 |\log\sigma|^\alpha$.} This inequality is deduced from an upper bound on the kinetic energy functional $E =  \frac 12 \int\!\rmd\bs\xi\, |\bs u (\bs\xi,t)|^2$, which in cylindrical coordinates $x=(x_1,x_2) = (z,r)$ takes the form
\[
E(t) =  \frac 12 \int\! \rmd x \,  2\pi x_2 |u(x,t)|^2\,,
\]
combined with the bounds
\begin{align}
\label{M0<}
M_0(t)  & = \int\! \rmd x\, \omega_\eps(x,t) \le \frac{C}{|\log\eps|}\,, \\ \label{M2<} M_2(t)  & = \int\! \rmd x \, x_2^2 \omega_\eps(x,t) \le \frac{C}{|\log\eps|}\,, \\ \label{E>} E(t) & \ge  E(0) - \frac{C\nu}{\eps^2|\log\eps|^2} \ge E(0) - \frac{C}{|\log\eps|^{2-\gamma}}\,,
\end{align}
with $\gamma \in (0,1)$ (for $t\in [0,T]$ and any $\eps$ small enough).

In Ref.~\cite{BruMar}, the upper bounds on $M_0$ and $M_2$ are easily deduced since $M_0(t) \le M_0(0)$, $M_2(t) = M_2(0)$ (in absence of external field, recall Eq.~\eqref{M2d=0}), and both $M_0(0)$ and $M_2(0)$ are bounded from above by $C/|\log\eps|$ in view of the assumption Eq.~\eqref{initialF} on the initial distribution. Concerning the lower bound on the energy (which is a conserved quantity in the Euler case in absence of external field), from the well known formula for the energy dissipation due to viscosity we have,
\begin{align}
\label{dotnu}
\dot E(t) & = -\nu \int_{\mathbb{R}^3}\rmd\bs\xi\,\sum_{i,j=1}^3 \left(\partial_{\xi_j} u_i(\bs\xi,t)\right)^2 = -\nu \int_{\mathbb{R}^3}\rmd\bs\xi\, |\bs \omega(\bs\xi,t)|^2 \nonumber \\
& =  -2\pi \nu \int\! \rmd x\, \omega_\eps^2(x,t) x_2  \ge -2\pi\nu M_2(t) \left\| \frac{\omega_\eps(t)}{x_2} \right\|_{L_\infty} \ge - \frac{C\nu}{\eps^2|\log\eps|^2}\,,
\end{align}
where we used Eqs. \eqref{M2<}, \eqref{maxpri}, and \eqref{initialF}. Therefore,
\begin{equation}
\label{viscosity_term}
-C\frac{ |\log\eps|^\gamma}{|\log\eps|^2}\le-\frac{C\nu}{\eps^2|\log\eps|^2}\le \dot E \le 0\,,
\end{equation}
by the assumption $\nu \le \eps^2 |\log\eps|^\gamma$, where $\gamma<1$. 

From what discussed above, we conclude that Lemma \ref{lem:4} is proved if we show that  Eqs.~\eqref{M0<}, \eqref{M2<}, and \eqref{E>} hold true also in the present case.  

The function $M_0(t)$ is non-increasing also in our case, and indeed we have already noticed the bound Eq.~\eqref{M0<}, see Eq.~\eqref{cons-mass}.

Concerning the variation of $M_2(t)$, we can apply Eq.~\eqref{weqF} with $f(x,t) = x_2^2$, so that, by Eq.~\eqref{M2d=0} (the contribution of the viscosity is zero),
\[
\dot M_2(t) = \int\! \rmd x \, \omega_\eps(x,t) \,2 x_2  F^\eps_2(x,t)\,.
\]
Therefore, from the Cauchy-Schwarz inequality and Eqs.~\eqref{ext_field}, \eqref{sing_Fc}, \eqref{sing_F}, \eqref{reg_F}, \eqref{picc_F}, and \eqref{cons-mass}, 
\[
\begin{split}
|\dot M_2(t)| & \le \sqrt{M_2(t)}\,  \sqrt{\int\!\rmd x\, \omega_\eps(x,t) F^\eps_2(x,t)^2} \\ & \le  \sqrt{M_2(t)} \,  \sqrt{\frac{C}{|\log\eps|^3} + \frac{\widehat C^2}{\eps^2|\log\eps|^2} \int_{|x_2-r_0| > \widehat r} \!\rmd x\, \omega_\eps(x,t)}\,,
\end{split}
\]
hence, by Eq.~\eqref{smt*},
\[
|\dot M_2(t)| \le C |\log\eps|^{-3/2} \sqrt{M_2(t)} \quad \forall\, t\in [0,T]\,,
\]
which implies $M_2(t) \le M_2(0) + C |\log\eps|^{-2}$ for any $t\in [0,T]$, so that Eq.~\eqref{M2<} follows, i.e.,
\begin{equation}
\label{M2<*} 
M_2(t) \le \frac{C}{|\log\eps|} \quad \forall\, t\in [0,T]\,.
\end{equation}

Finally, concerning the energy $E(t)$, its time derivative is
\[
\dot E(t) = \dot E_\nu(t) + \dot E_{F^\eps}(t)\,,
\]
where $\dot E_\nu(t)$ is the variation due to the viscosity, which has been already treated in Eq.~\eqref{dotnu}. To compute the variation $\dot E_{F^\eps}(t)$ due to the external field, we adapt the strategy developed in Ref.~\cite[Appendix A]{ButCavMar}, where the Euler case with regular external field $F^\eps=\widetilde F^\eps$ is considered. We write $u(x,t) = x_2^{-1} \nabla^\perp\Psi(x,t)$, where $\Psi(x,t)$ is the stream function,
\[
\Psi(x,t) = \int\!\rmd y\, S(x,y)\, \omega_\eps(y,t)\,,
\]
with $S(x,y)$ the Green function,
\begin{equation}
\label{S}
S(x,y) := \frac{x_2y_2}{2\pi} \int_0^\pi\!\rmd\theta\, \frac{\cos\theta}{\sqrt{|x-y|^2 + 2x_2y_2(1-\cos\theta)}}\,.
\end{equation}
Therefore, the energy reads (see, e.g., Refs.~\cite{BCM00,F}),
\[
E = \pi \int\! \rmd x\, \Psi(x,t) \, \omega_\eps(x,t) = \pi \int\!\rmd x \int\!\rmd y\, S(x,y)\, \omega_\eps(x,t) \, \omega_\eps(y,t)\,,
\]
and, by Eq.~\eqref{weqF}, the contribution to $\dot E(t)$ due to the external field is\footnote{Observe that $u\cdot \nabla \Psi=0$.}
\[
\dot E_{F^\eps}(t) = \pi \int\!\rmd x\, \omega_\eps(x,t) \, (F^\eps \cdot \nabla \Psi + \partial_t \Psi)(x,t)\,.
\]
We now recall that in Ref.~\cite[Appendix A]{ButCavMar} the following identity for $\dot E_{F^\eps}(t)$ holds true,
\[
\dot E_{F^\eps}(t)  = 2 \pi \int\!\rmd x\, \omega_\eps(x,t) (F^\eps \cdot \nabla \Psi)(x,t) \, .
\]
According to the decomposition Eq.~\eqref{ext_field} of $F^\eps$, we then write $\dot E_{F^\eps}(t) = \dot E_{\widehat F^\eps}(t) + \dot E_{\widetilde F^\eps}(t) + \dot E_{\bar F^\eps}(t)$ and analyze separately these terms.

As already noticed, $\widetilde F^\eps$ has the same properties of the external field considered in Ref.~\cite{ButCavMar}, so that the analysis of Ref.~\cite[Appendix A]{ButCavMar} applies to $\dot E_{\widetilde F^\eps}(t)$, obtaining
\[
\big|\dot E_{\widetilde F^\eps}(t)\big| \le \frac{C}{|\log\eps|^2}\,.
\]

To analyze the remaining terms, we shall use the following estimate on $u(x,t)$,
\begin{equation}
\label{uinf<}
\|u(\cdot, t)\|_{L^\infty} \le \frac{C}{\eps |\log\eps|} \quad \forall\, t\in [0,T]\,,
\end{equation}
which follows form Eqs.~\eqref{u_piccolo}, using Eq.~ \eqref{maxpri} together with Eqs.~\eqref{MgammaF} and \eqref{initialF}, Eq.~\eqref{cons-mass}, and the estimate $M_2 = \| x_2^2 \,\omega_\eps \|_{L^1} \le C/|\log\eps|$ proved above. Therefore, as $|\nabla \Psi(x,t)|=|\nabla^\perp \Psi(x,t)| = x_2 |u(x,t)|$, by Eqs.~\eqref{sing_Fc}, \eqref{sing_F}, and the Cauchy-Schwarz inequality,
\[
\big|\dot E_{\widehat F^\eps}(t)\big| \le \frac{C\sqrt{M_2(t)}}{\eps^2|\log\eps|^2} \sqrt{\int_{|x_2-r_0|>\widehat r} \! \rmd x\, \omega_\eps(x,t)} \le C_\ell \eps^\ell \quad \forall\, t\in [0,T] \quad\forall\, \ell>0\,,
\]
where we used Eqs.~\eqref{smt*} and \eqref{M2<*} in the last inequality. Similarly, by Eq.~\eqref{picc_F},
\[
\big|\dot E_{\bar F^\eps}(t)\big| \le \frac{C\eps^{\beta-1}}{|\log\eps|} \int\!\rmd x\, \omega_\eps(x,t) x_2 =  \frac{C\eps^{\beta-1}}{|\log\eps|^2} B_{\eps,2}(t) \le \frac{C\eps^{\beta-1}}{|\log\eps|^2}\,,
\]
where we used that $B_{\eps,2}^*(t) \le C$ and $|B_{\eps,2}(t)-B_{\eps,2}^*(t)|$ is vanishing for $\eps\to 0$.

In conclusion, 
\[
\big|\dot E_{F^\eps}(t)\big| \le \frac{C}{|\log\eps|^2} \ll \frac{C}{|\log\eps|^{2-\gamma}}
\]
and also Eq.~\eqref{E>} holds in the present case. The lemma is thus proved.
\qed


\begin{thebibliography}{99}

\bibitem{AmS89} Ambrosetti, A., Struwe, M.:  Existence of steady rings in an ideal fluid. Arch. Ration. Mech. Anal. {\bf 108}, 97--108 (1989)


\bibitem{BCM00} Benedetto, D., Caglioti, E., Marchioro, C.: On the motion of a vortex ring with a sharply concentrate vorticity. Math. Meth. Appl. Sci. \textbf{23}, 147--168 (2000)

\bibitem{BruMar} Brunelli, E., Marchioro, C.: Vanishing viscosity limit for a smoke ring with concentrated vorticity. J. Math. Fluid Mech. \textbf{13}, 421--428 (2011)

\bibitem{ButCavMar} Butt\`a, P., Cavallaro, G., Marchioro, C.: Global time evolution of concentrated vortex rings. Z. Angew. Math. Phys. \textbf{73}, Article number 70 (2022)

\bibitem{BuM1} Butt\`a, P., Marchioro, C.: Long time evolution of concentrated Euler flows with planar symmetry. SIAM J. Math. Anal. \textbf{50}, 735--760 (2018)

\bibitem{BuM2} Butt\`a, P., Marchioro, C.: Time evolution of concentrated vortex rings. J. Math. Fluid Mech. \textbf{22}, Article number 19 (2020)

\bibitem{CavMar21}  Cavallaro, G., Marchioro, C.:  Time evolution of vortex rings with large radius and very concentrated vorticity.   J. Math. Phys. \textbf{62},  053102,  20 pp. (2021)

\bibitem{CS} Cetrone, D.,  Serafini, G.: Long time evolution of fuids with concentrated vorticity and convergence to the point-vortex model.  Rendiconti di Matematica e delle sue applicazioni \textbf{39}, 29--78 (2018)


\bibitem{Fe-Sv} Feng, H.,  \v{S}ver\'{a}k, V.: On the Cauchy problem for axi-symmetric vortex rings. Arch.
Ration. Mech. Anal.  \textbf{215},  89--123   (2015)

\bibitem{Fr70} Fraenkel, L.E.: On steady vortex rings of small cross-section in an ideal fluid. Proc. Roy. Soc. Lond. A. \textbf{316}, 29--62 (1970)

\bibitem{FrB74} Fraenkel, L.E., Berger M.S.: A global theory of steady vortex rings in an ideal fluid. Acta Math. \textbf{132}, 13--51 (1974)

\bibitem{F} Friedman, A.: Variational Principles and Free-Boundary Problems. Wiley, New York 1982

\bibitem{Gal11} Gallay, T.: Interaction of vortices in weakly viscous planar flows. Arch. Ration. Mech. Anal. \textbf{200}, 445--490 (2011)

\bibitem{Gal13} Gallay, T., \v{S}ver\'{a}k, V.: Remarks on the Cauchy problem for the axisymmetric Navier-Stokes equations. Confluentes Math. \textbf{7}, 67--92  (2015)

\bibitem{Gal12} Gallay, T., \v{S}ver\'{a}k, V.: Uniqueness of axisymmetric viscous flows originating from circular vortex filaments. Ann. Sci. \'{E}c. Norm. Sup\'{e}r. \textbf{52}, 1025--1071 (2019)

\bibitem{HLM} Hientzsch, L.E., Lacave, C., Miot, E., Dynamics of several point vortices for the lake equations, preprint arXiv:2207.14680 (2022).

\bibitem{Lad} Ladyzhenskaya, O.A.: Unique solvability in large of a three-dimensional Cauchy problem for the Navier–Stokes equations in the presence of axial symmetry. Zapisky Nauchnych Sem. LOMI \textbf{7}, 155--177 (1968)


\bibitem{Mar90} Marchioro, C.: On the vanishing viscosity limit for two-dimensional Navier-Stokes equations with singular initial data. Math. Meth. Appl. Sci. \textbf{12},  463--470 (1990)

\bibitem{Mar98} Marchioro, C.: On the inviscid limit for a fluid with a concentrated vorticity.
Commun. Math. Phys. \textbf{196}, 53--65 (1998)

\bibitem{Mar99} Marchioro, C.: Large smoke rings with concentrated vorticity. J. Math. Phys. \textbf{40}, 869--883 (1999)

\bibitem{Mar07} Marchioro, C.: Vanishing viscosity limit for an incompressible fluid with concentrated vorticity. J. Math. Phys. \textbf{48}, 065302, 16 pp. (2007)

\bibitem{MarNeg} Marchioro, C., Negrini, P.:  On a dynamical system related to fluid mechanics. NoDEA Nonlinear Diff. Eq. Appl., \textbf{6}, 473--499 (1999)

\bibitem{MaP93} Marchioro, C., Pulvirenti, M.: Vortices and localization in Euler flows. Commun. Math. Phys. \textbf{154}, 49--61,(1993).

\bibitem{MaP94} Marchioro, C., Pulvirenti, M.: Mathematical theory of incompressible non-viscous fluids. Applied mathematical sciences vol.~96, Springer-Verlag, New York, 1994

\bibitem{ShL92} Shariff, K., Leonard, A.: Vortex Rings. Annu. Rev. Fluid  Mech. \textbf{24}, 235--279 (1972)

\bibitem{UY} Ukhovskii, M., Yudovitch, V.: Axially symmetric flows of ideal and viscous fluids filling the whole space. J. Appl. Math.
Mech. \textbf{32}, 52--69 (1968)


\end{thebibliography}
\end{document}